\documentclass[a4paper,12pt]{article}
\usepackage{amsmath,amsfonts,amssymb,amsthm,graphicx,ytableau,mathtools,epigraph}

\newtheorem{thm}{Theorem}[section]
\newtheorem{lem}[thm]{Lemma}
\newtheorem{prop}[thm]{Proposition}
\newtheorem{cor}[thm]{Corollary}
\newtheorem{rem}[thm]{Remark}

\input epsf
\title{Generic numerical semigroups}
\author{Roland Bacher}
\begin{document}
\maketitle

\begin{abstract} The use of compositions simplifies some aspects of the
theory of numerical semigroups. We illustrate this by giving a
new proof for the asymptotic number $C((1+\sqrt{5})/2)^g$
of numerical semigroups of genus $g$ and by describing
the constant $C$ explicitly\footnote{
 Math. class:  
 20M14, 05A16.
}.
\end{abstract}





\section{Introduction}

\epigraph{..., 
chacun appelant id\'ees claires celles qui sont au m\^eme degr\'e de
confusion que les siennes propres.}{\textit{Marcel Proust}}


A \emph{numerical semigroup}\footnote{Numerical semigroups are a hot topic 
during outbreaks of contagious diseases: You don't want 
to take change at grocery stores. 
(This paper was largely written
during the Corona-virus lock-down.)} 
is a subgroup $S=S+S$ of the
additive semigroup $\mathbb N=\mathbb N+\mathbb N$
such that the complementary set $\mathbb N\setminus S$ is finite. 
The \emph{embedding dimension} $e=e(S)$ of $S$ is the minimal
cardinal of a generating set. The smallest non-zero
element $m=m(S)=\min(S\setminus\{0\})$ of $S$ 
is the \emph{multiplicity} of $S$. The finite set
$G=G(S)=\mathbb N\setminus S$ is the \emph{set of gaps} and
the number $g=g(S)=\sharp(G)$ of elements in $G$ is the \emph{genus}
of $S$. The \emph{Frobenius number} $f=f(S)=\max(G)$ is the maximal
element of $G$.

Throughout the rest of the paper, the letter $m$ will always denote
the multiplicity $\min(S\setminus\{0\})$ of a numerical semigroup $S$.

For $j=0,\dots,m-1$, we set
\begin{align}\label{defxj}
  x_j&=x_j(S)=\sharp\{G\cap j+m\mathbb Z\}\ .
\end{align}
We have $x_0=0$ and $x_1,\dots,x_{m-1}\geq 1$. The vector
$(x_1,\ldots,x_{m-1})$ is called the \emph{Kunz coordinate vector}
or simply the \emph{Kunz vector} of $S$. It is also called
the \emph{Ap\'ery tuple} of $S$, see e.g. \cite{Kaplan}.
The trivial identity
$$g=\sum_{j=1}^{m-1} x_j$$
shows that $x_1+x_2+\cdots+ x_{m-1}$ is a composition of $g$
into $m-1$ parts. The Frobenius number $f$ of $S$ is given
by 
$$\max_{j\in\{1,\ldots,m-1\}} j+m(x_j-1)\ .$$
The Frobenius number $f=l+m(x_l-1)$ is equivalently defined
in terms of the index and the value of the last maximal part $x_l$
of the composition $x_1+\cdots +x_{m-2}$ with parts defined by (\ref{defxj}).

We call the composition $x_1+\cdots+x_{m-1}$ the \emph{composition of  $S$}.
It is closely related to the \emph{Ap\'ery set}
$$\mathrm{Ap}(m,S)=\{0,x_1m+1,x_2m+2,\ldots,x_{m-1}m+m-1\}$$
consisting of minimal representatives in $S$ for classes modulo
the multiplicity $m=\min(S\setminus\{0\})$ of $S$.

\begin{rem}\label{remgroupalgebra}
By a fortunate coincidence, the letter $m$ (denoting the multiplicity
of $S$) is also well-suited for denoting classes modulo $m\mathbb Z$
of $\mathbb Z$. The composition $x_1+\cdots+x_{m-1}$
can be considered as the image in the group algebra
$\mathbb Z[\mathbb Z/m\mathbb Z]$ of the characteristic function
$\sum_{\gamma\in \mathbb N\setminus S}[\gamma]$ for the gap-set 
in the group algebra $\mathbb Z[\mathbb Z]$ of $\mathbb Z$.
The trivial character applied to the characteristic function of
$\mathbb N\setminus S$ yields the genus of $S$.
\end{rem}

It is easy to show that a numerical semigroup is determined by
its composition with parts defined by (\ref{defxj}). We will recall a 
well-known characterisation of compositions associated to numerical
semigroups. We can thus replace the tree-structure underlying numerical
semigroups by the (essentially) linear structure of 
compositions when studying certain properties of numerical semigroups. 
This reduces the study of asymptotic growth
(for numbers of numerical semigroups of given genus)
to elementary considerations boiling often down
to ``generatingfunctionology'', after making use of simple
geometric properties of (integers contained in) intervals.

We use this approach in a new proof for asymptotics 
conjectured by Maria Bras-Amor\'os in \cite{Bras}
and proven by Alex Zhai in \cite{Zhai} (a pleasant overview of this
topic describing the gist of Zhai's proof
is given in \cite{Kaplan})
for the number $n(g)$ of semigroups of genus $g$:
\begin{thm}\label{thmupperboundng} There exists a constant $C$ such that 
\begin{align}\label{formulaasymptsemigrps}
\lim_{g\rightarrow\infty}
\frac{n(g)}{\omega^g}&=C
\end{align}
where $\omega=\frac{1+\sqrt{5}}{2}=1.61803398874989\ldots$
is the golden number.

The constant $C$ is given by 
\begin{align}\label{formulaforC}
C&=\frac{5+\sqrt{5}}{10}\left(1+\tilde C(\omega^{-1})\right)
\end{align}
where $\tilde C(q)$ is the generating function
\begin{align}\label{formulafortildeC}
\tilde C(q)&=\sum_{g=3}^\infty \tilde c_gq^g
\end{align}
enumerating the number $\tilde c_g$ of numerical semigroups
$S$ of genus $g\geq 3$
satisfying the identity $f=3m-1$ linking their Frobenius number
$f=\max(\mathbb N\setminus S)$ with their multiplicity
$m=\min(S\setminus\{0\})$.

The series $\tilde C(q)$ defines a holomorphic function in an open
disc of radius strictly larger than $\omega^{-1}$.
\end{thm}

\begin{rem} The convergency radius of $\tilde C$ is
  at most equal to the positive real root $0.659982\ldots$
  of $1-q^3-2q^4-2q^5-q^6$, cf. Theorem \ref{thmlowerboundtildeC}. It is thus
  only slightly larger than
  $\omega^{-1}\sim 0.618$. Using this approach for determining a good
  numerical approximation of the constant $C$ involved in
  (\ref{formulaasymptsemigrps}) is thus quite tricky.
\end{rem}

\begin{rem} Formula (\ref{formulaasymptsemigrps}) of Theorem
  \ref{thmupperboundng} is sometimes given with $\omega^g$
  replaced by the $g$-th Fibonacci number. 
  This changes the value
  of the associated constant $C'$  which depends also (up to a power
  of $\omega$) on the
  convention for indices of Fibonacci numbers.
\end{rem}

Throughout the paper we use always $\omega$ for the golden number
$\omega=\frac{1+\sqrt{5}}{2}$ with
multiplicative inverse $\omega^{-1}=\frac{\sqrt{5}-1}{2}$.

The overall structure of our proof of Theorem \ref{thmupperboundng}
and of the proof in \cite{Zhai} are similar:
Numerical semigroups whose Frobenius
numbers are much larger than twice their multiplicity
can be neglected when considering asymptotics.
Zhai proves this by analysing the tree-structure of numerical
semigroups. We use instead compositions with parts defined by
(\ref{defxj}) which we call numerical semigroup compositions 
(called Kunz coordinate vectors or
special cases of so-called Ap\'ery sets by other authors)
or NSG-compositions for short. The linear nature
of compositions makes their study fairly elementary: It is
essentially equivalent to the geometry of closed real intervals,
endowed with partial actions by real reflections.
Transfer matrix techniques and elementary properties of
series expansions for holomorphic functions
complete the proof.

Sections \ref{sectNSGcomp}-\ref{sectmainproof} are devoted to the proof of Theorem
\ref{thmupperboundng}.

Section \ref{sectNSGcomp} introduces numerical semigroup compositions
or NSG-compositions for short.

Section \ref{sectmax2} illustrates the notion of NSG-compositions
by describing all NSG-compositions of maximum at most $2$,
a result also contained in \cite{Zhai}.

Sections \ref{sectgeneralized} and \ref{secttree} are digressions
describing generalised compositions, algorithmic aspects and
the tree-structure of NSG-compositions.

Section \ref{sectoutline} contains a rough outline for the proof of Theorem \ref{thmupperboundng}.

Section \ref{sectpivot} introduces pivot-factorisation, our main tool
for obtaining enumerative results on NSG-compositions.

Section \ref{sectgen} recalls a few facts concerning generating series and
growth-rates.

Section \ref{sectweakly} defines weak admissibility for compositions.
This is used in Section \ref{sectmax6} for obtaining upper bounds on
the growth-rate of NSG-compositions of maximum at least $6$.

NSG-compositions of maximum $5$ and $4$ are treated in Sections
\ref{sectmax5} and \ref{sectmax4}.

Section \ref{sectmax3} describes the Combinatorics of NSG-compositions
of maximum $3$.

Section \ref{sectmainproof} completes the proof of Theorem 
\ref{thmupperboundng} by giving upper bounds on the growth-rate of
NSG-compositions of maximum $3$ ending with a maximal part
enumerated by the series $\tilde C$ occurring in
Formula (\ref{formulafortildeC}).

Sections \ref{sectNSGcomp3}-\ref{sectprobas} outline
some (sometimes only conjectural) combinatorial or probabilistic aspects
of numerical semigroups and certain types of compositions.

\section{Numerical semigroup-compositions}\label{sectNSGcomp}

We start this Section with a justification of our (non-standard)
terminology. Numerical semigroup-compositions are equivalent to
Kunz coordinate vectors, an unfortunate
choice of terminology in our opinion: there is no underlying vector
space and these ``vectors'' encode simply
finite sequences of strictly positive integers
indexed by $1,2,\ldots,m-1$ (representing all non-zero classes of
$\mathbb Z/m\mathbb Z$) summing up to the genus.

Our next result, which is folklore
(see for example \cite{Kaplan}), shows that numerical
semigroups are encoded by their compositions:

\begin{prop}\label{propSdefdbycomp}
A numerical semigroup $S$ associated to a composition 
$x_1+\cdots +x_{m-1}$ with parts defined by formula (\ref{defxj})
is uniquely determined by the formula 
\begin{align}\label{formulaforSbycomp}
S&=\cup_{j=0}^{m-1} (j+mx_j+m\mathbb N)
\end{align}
\end{prop}
using the convention $x_0=0$.

A \emph{numerical semigroup composition} (or an 
\emph{NSG-composition}, for short) is a composition 
$x_1+\cdots+x_{m-1}$ (with an omitted trivial part $x_0=0$)
defining a numerical semigroup
by (\ref{formulaforSbycomp}). We identify henceforth numerical
semigroups with their NSG-compositions.

\begin{proof}[Proof of Proposition \ref{propSdefdbycomp}] 
A composition $x_1+\cdots+x_{m-1}$ associated to a numerical 
semigroup $S$ determines
the multiplicity $m$ of $S$ (by adding $1$ to the number $m-1$
of its summands). Every numerical semigroup associated to 
$x_1+\cdots+x_{m-1}$ has thus the same multiplicity $m$.
Since $m$ is an element of $S$, the intersection of $S$
with an arithmetic progression $j+m\mathbb N$ (for $j$ in $\{
1,\ldots,m-1\}$) lacks consecutive initial values of the arithmetic
progression $j+m\mathbb N$. This shows that $S\cap(j+m\mathbb Z)$ is given
by $j+x_jm+m\mathbb N$ for every $j\in\{1,\ldots,m-1\}$.
Since $m\in S$ implies $S\cap m\mathbb N=m\mathbb N$,
the composition $x_1+\cdots+x_{m-1}$ defines the intersection
of $S$ with $j+m\mathbb Z$ for all $j$. This determines
$S$ uniquely.
\end{proof}

\begin{cor} There are at most $2^{g-1}$ numerical semigroups of genus 
$g$.
\end{cor}

\begin{proof} A strictly positive integer $g>0$
has $2^{g-1}$ compositions\footnote{Write the word $1^g$ 
consisting of $g$ identical letters
$1$. Starting with the second occurrence of $1$, decide for each letter $1$ 
if you add it to the precedent letter or if you do nothing. These $2^{g-1}$
possible choices produce all possible sequences of strictly positive 
integers adding up to $g$. Equivalently, write all 
$2^{g-1}$ words of $0\{0,1\}^{g-1}$ of length $g$, with letters in $\{0,1\}$
and first letter $0$: 
Consider lengths of factors for such a word  
after rewriting it in terms of $0,01,011,
0111,\ldots$.

For a proof with generating series observe that $(q/(1-q))^n$ counts
the number of compositions with exactly $n$ parts. We get the result
by the identities
$\sum_{n=1}^\infty (q/(1-q))^n=(q/(1-q))\circ(q/(1-q))=q/(1-2q)$
for the generating series enumerating all compositions of
strictly positive integers.}.
\end{proof}
 
The next result is well-known (see for example Proposition 9 in
\cite{Kaplan})
and describes the set of all NSG-compositions:
\begin{thm}\label{thmsemigrcomp}
A composition $x_1+\cdots +x_{m-1}$ is the composition
of a numerical semigroup of multiplicity $m$ if and only if 
we have the inequalities 
\begin{align}\label{fundeqsg}
\begin{array}{rl}x_{s+t}&\leq x_s+x_t,\\
x_{m-s-t}&\leq x_{m-s}+x_{m-t}+1
\end{array}
\end{align}
for all $s,t$ in $\{1,\ldots,m-2\}$ such that $s+t<m$.
\end{thm}

The inequalities given by (\ref{fundeqsg}) are henceforth called
\emph{NSG-inequalities}.

\begin{proof}[Proof of Theorem \ref{thmsemigrcomp}]
Let $x_1+\ldots+x_{m-1}$ be a composition to which we add 
a trivial part $x_0=0$.
We define a subset $S$ of $\mathbb N$ by setting
$$S=\bigcup_{j=0}^{m-1}(j+x_jm+m\mathbb N)$$
(cf. formula (\ref{formulaforSbycomp}) of Proposition \ref{propSdefdbycomp}).
We have to show that $S$ is a numerical semigroup if and only if
all NSG-inequalities (\ref{fundeqsg}) hold: 
Given two elements $a$ and $b$ of $S$ we
consider their representatives $s,t$ modulo $m$ in $\{0,\ldots,m-1\}$.
By construction of $S$, there exist two natural integers 
$\alpha$ and $\beta$ such that $a=s+x_sm+\alpha m$ and $b=
t+x_tm+\beta m$. We have thus $c=a+b=s+t+(x_s+x_t)m+(\alpha+\beta)m$.

If $s+t<m$ we have $u=s+t$ in $\{1,\ldots,m-1\}$
representing $s+t$ modulo $m$. We get thus 
\begin{align*}
c&=s+t+(x_s+x_t)m+(\alpha+\beta)m\\
&=u+x_um+(\alpha+\beta+x_s+x_t-x_u)m
\end{align*}
and $\alpha+\beta+x_s+x_t-x_u$ is always a natural integer if 
and only if all inequalities of the first line in (\ref{fundeqsg})
hold.

If $s+t\geq m$, we get $u=s+t-m$ for $u$ in $\{0,\ldots,m-1\}$ 
and we have 
\begin{align*}
c&=(s+t-m)+(x_s+x_t+1)m+(\alpha+\beta)m\\
&=u+x_um+(\alpha+\beta+x_s+x_t+1-x_u)m
\end{align*}
and $\alpha+\beta+x_s+x_t+1-x_u$ is always in $\mathbb N$
if and only if we have the inequalities of the second line in 
(\ref{fundeqsg}).
\end{proof}

\begin{rem} The NSG-inequalities (\ref{fundeqsg}) can be
  rewritten as
  $$x_i+x_j\geq x_{i+j\pmod m}+c(i,j)$$
  where $$c(i,j)=\left\lbrace\begin{array}{ll}
      0\quad&\hbox{if }i+j<m,\\
      1&\hbox{otherwise}
    \end{array}\right.$$
  is the $2$-cocycle of $H^2(\mathbb Z/m\mathbb Z,\mathbb Z)$
  corresponding to the exact sequence
$$0\longrightarrow \mathbb Z\longrightarrow \mathbb Z
\longrightarrow \mathbb Z/m\mathbb Z\longrightarrow 0
$$
defining $\mathbb Z$ as a central extension of $\mathbb Z/m\mathbb Z$ by
$\mathbb Z$. This observation can of course be explained
by Remark \ref{remgroupalgebra}.
\end{rem}

\subsection{Parameters}

We describe without proofs how to recover
basic parameters of a numerical semigroup $S
=m\mathbb N\bigcup\left(\cup_{j=1}^{m-1}\lbrace j+m(x_j+
  \mathbb N)\rbrace\right)$
from its NSG-composition $x_1+\ldots +x_{m-1}$.

The multiplicity $m$ of $S$ is easily obtained from the
number $m-1$ of parts of $x_1+\ldots+x_{m-1}$. Similarly, 
the genus is the sum $\sum_{i=1}^{m-1}x_i$ of all parts.

Minimal generators for $S$ are given by the multiplicity $m$ and by integers
$j+x_jm,\ j\in \{1,\ldots,m-1\}$
such that all NSG-inequalities (\ref{fundeqsg}) are strict 
for $s,t\in\{1,\ldots,m-1\}$ with $s+t$ in $\{j,j+m\}$.

The Frobenius number $f=\max(\mathbb N\setminus S)$
of a numerical semigroup with composition $x_1+\ldots+x_{m-1}$
is given by
\begin{align}\label{formulafrobenius}
f&=\max(x_1,\ldots,x_{m-1})m-m+\max(\{j\ \vert\ x_j=\max(x_1\ldots,x_{m-1})\})
\end{align}
or equivalently by $m(x_l-1)+l$ where $x_1,\ldots,x_{l-1}\leq x_l>x_{l+1},
\ldots,x_{m-1}$ (i.e. $x_l$ is the last summand of maximal value).
We have the inequalities
\begin{align*}
m(\max(x_1,\ldots,x_{m-1})-1)<f<m\max(x_1,\ldots,x_{m-1})
\end{align*}
and the equality $\max(x_1,\ldots,x_{m-1})=\lceil f/m\rceil$.

\section{NSG-compositions with maximum $2$}
\label{sectmax2}

Results of this Section are well-known, see for example \cite{Zhai}.

\begin{prop}\label{propparts12} All compositions with parts
in $\{1,2\}$ are NSG-compositions. 

The generating function for the number of compositions 
$x_1+\cdots+x_{m-1}$ with genus $g$ and all parts $x_j$ in $\{1,2\}$
is the generating function
$$\sum_{n=0}^\infty F_gq^g=\frac{1}{1-q-q^2}$$
of Fibonacci numbers
$F_g=\frac{5+\sqrt{5}}{10}\left(\frac{1+\sqrt{5}}{10}\right)^g+
\frac{5-\sqrt{5}}{2}\left(\frac{1-\sqrt{5}}{2}\right)^g$
defined recursively by $F_0=1,F_1=1,F_2=2,\ldots,F_n=F_{n-1}+F_{n-2},\dots$.
\end{prop}

\begin{proof}
Compositions with all parts $x_j$ in $\{1,2\}$
satisfy obviously all NSG-inequalities (\ref{fundeqsg}) 
of Theorem \ref{thmsemigrcomp} and are thus NSG-compositions. 

The generating series for compositions with all parts in $\{1,2\}$
is given by $\sum_{n=0}^\infty(q+q^2)^n=1/(1-(q+q^2))$.
\end{proof}

\begin{rem} Two other proofs for the generating series
  enumerating compositions with all parts in $\{1,2\}$ are as follows:

There is a unique such composition of genus $0$ or $1$.
Every such composition of genus $\geq 2$ is obtained by adding a final
part $1$ to such a composition of genus $g-1$ or by adding a final part
$2$ to such a composition of genus $g-2$. Numbers of such compositions
satisfy thus the initial conditions and recurrence relations
of Fibonacci-numbers.

A composition of genus $g=\sum_{j=1}^{m-1} x_j$ with all parts in $\{1,2\}$
and with multiplicity $m=g-k$ (for $k$ in $\{0,\ldots,\lfloor g/2\rfloor\}$)
has $k$ parts equal to $2$ and $g-2k$ parts equal to 
$1$. There are thus $\sum_{k\geq 0}{g-k\choose k}$ such compositions.
The easy identity
$$\sum_{k=0}^{\lfloor g/2\rfloor}{g-k\choose k}=F_g$$
for $F_0=F_1=1$ and $F_n=F_{n-2}+F_{n-1}$ the sequence of Fibonacci numbers
ends the proof.
\end{rem}


\begin{rem} A bijection between
$\{0,1,\ldots,F_g-2,F_g-1\}$ and compositions of $g$ with parts in
$\{1,2\}$ can be constructed as follows:
The Zeckendorff expansion  
$\epsilon_g\epsilon_{g-1}\ldots \epsilon_2\epsilon_1$ of length $g$
(defined by $n=\sum_{j=1}^g\epsilon_j F_j$
with $\epsilon_j\in\{0,1\},\ \epsilon_{j+1}\epsilon_j=0$ for
all $j$) of an integer $n<F_g$
starts with $\epsilon_g=0$ and contains only isolated digits $1$.
Rewriting it in terms of $0$ and $01$ and considering 
lengths of factors yields the composition
of $g$ associated to $n$.
\end{rem}

We end this Section with a digression on the following
well-known properties (not directly related to the
topic of the paper) of compositions:
\begin{rem}
Compositions of $n$ having only parts $\leq 2$ are equinumerous 
with compositions of all integers up to $n$ having only parts $\geq 2$:
Given a composition $x_1+\ldots+x_k$ of $n$ with $x_j\in\{1,2\}$,
remove the final block (which is empty if $x_k=2$)
of consecutive parts $1$, factor the resulting word 
$x_1x_2\ldots x_l$ with factors in $\{1\}^*2$ and replace
$1^k2$ by $k+2$. The resulting word corresponds to a composition of an integer
$\leq n$ with summands $\geq 2$.

Equivalently, compositions of $n$ with all parts in $\{1,2\}$ correspond
to compositions of $n+2$ with all parts $\geq 2$. (Add a final part $2$
before applying the above algorithm.)

Combining the two previous bijections yields a bijection between
all $F_n$ compositions with parts $\geq 2$ of integers up to $n$
and all $F_n$ compositions with parts $\geq 2$ of the integer $n+2$.

Subtracting $1$ from the first part of partitions with all parts $\geq 2$
shows  finally that there are also $F_n$
partitions of $n+1$ with arbitrary first part
and all other parts $\geq 2$.
Such partitions encode Zeckendorff expansions of integers in $\{F_{n+1},\ldots,
F_{n+2}-1\}$ by associating to $x_1+\ldots +x_k$
the word $\epsilon_1\epsilon_2\ldots \epsilon_{n+1}$
(ending with $\epsilon_{n+1}=1$) obtained by replacing $x_k$ by $0^{k-1}1$.
The composition $1+3+2$ for example corresponds to $\epsilon_1\ldots
\epsilon_6=1\vert 001\vert 01$ encoding the integer
$\sum_{j=1}^6\epsilon_jF_j=
F_1+F_4+F_6=1+5+13=19$.
\end{rem}

\section{Generalised compositions}\label{sectgeneralized}

Let $\mathcal N\not\subset\{0\}$ be a non-trivial submonoid of
the additive monoid $(\mathbb N,+)$. 
We associate to a strictly positive element $M$ of $\mathcal N$ 
the cardinals $x_0,x_1,\ldots x_{M-1}$ in $\mathbb N\cup\{\infty\}$
counting the numbers
$$x_j=\sharp\left((j+M\mathbb N)\setminus \mathcal N\right)
\in\mathbb N\cup \{\infty\},\ j=0,\ldots,M-1$$
of elements in the complement (gap-set) $\mathbb N\setminus\mathcal N$
intersecting congruence classes modulo $M$. 

This defines a generalised composition
$x_0+x_1+\cdots+x_{M-1}$ with parts in $\mathbb N\cup\{\infty\}$ and
(perhaps infinite) sum $\sum_{j=0}^{M-1}x_j=\sharp\{\mathbb N\setminus
\mathcal N\}$ counting
the number of elements in the gap-set $\mathbb N\setminus \mathcal N$
of $\mathcal N$.

$\mathcal N$ is a semigroup (i.e. $\mathcal N$ contains $
0$) if and only if $x_0=0$. We work henceforth only with semigroups 
and omit the trivial summand $x_0=0$.

A semigroup $\mathcal N$ is a numerical semigroup if and only if
all parts $x_1,\ldots,x_{M-1}$ are natural integers.

The occurrence of infinite parts among $x_1,\ldots,x_{M-1}$
is equivalent to the existence of a divisor $d>1$ of $M$ such that
$\mathcal N=d\mathcal N'$ with $\mathcal N'$ a numerical semigroup.
We have then $x_j<\infty$ if and only if $d$ divides $j$.

All parts $x_1,\ldots,x_{M-1}$ are 
strictly positive (with $\infty$ being strictly positive by convention)
if and only if $M$ is the minimal non-zero element
$\min(\mathcal N\setminus\{0\})$ of $\mathcal N$.

The generalised composition $x'_1+\cdots+x'_{M'-1}$ associated to 
another non-zero element $M'$ of $\mathcal N$ is defined 
by considering the smallest natural integer $x'_j$ 
such that there exists $i\in\{0,\ldots,M-1\}$ with $x'_jM'+j=yM+i$
for $y\geq x_i$. We set $x'_j=\infty$ if $\mathcal N$ contains 
no elements congruent to $j$ modulo $M'$.
(It is also possible to compute $x'_1,\ldots,x'_{M'}$ using the 
algorithm sketched in Remark \ref{remalgo} below with generators
$\{M,x_1M+1,\ldots,x_{M-1}M+M-1\}\cap \mathbb N$.) 

Theorem \ref{thmsemigrcomp} holds (except for the assertion concerning 
the multiplicity which is at most equal to $M$) for generalised compositions
after extending the NSG-inequalities (\ref{fundeqsg}) to
$\mathbb N\cup\{\infty\}$ by considering $\infty$ as a maximal
element.

\begin{rem}\label{remalgo}
Generalised compositions have interesting algorithmic aspects.
They can easily be computed if $\mathcal N=\sum_{g\in \mathcal G}\mathbb N g$ 
(only finite sums are considered if $\mathcal G$ is infinite)
is defined in terms of a set $\mathcal G\subset \mathbb N\setminus\{0\}$ 
of non-zero generators: 

Choose an element $M$ in $\mathcal G$
(the choice $M=\min(\mathcal G)$ is optimal).

For $j=1,\ldots,M-1$, set
$x_j=\infty$ if $\mathcal G\cap \left(j+M\mathbb Z\right)=\emptyset$ and 
  $$x_j=\min_{g\in \mathcal G,\ g\equiv j\pmod M}\frac{g-j}{M}$$
otherwise.

Iterate the following loop until stabilisation:
For $j=1,\ldots,M-1$ replace $x_j$ by
$$\min\left(\{x_j\}\cup\left(\bigcup_{i=1}^{\lfloor j/2\rfloor}\{x_i+x_{j-i}\}
\right)\cup\left(\bigcup_{i=1}^{\lfloor (m-j)/2\rfloor}\{1+x_{j+i}+x_{m-i}\}
\right)\right)\ .$$
\end{rem}

\section{The tree of numerical semigroups}\label{secttree}

Numerical semigroups have a natural tree-structure:
Adding the Frobenius number 
$f=\max(\mathbb N\setminus S)$ 
to a numerical semigroup $S$ of strictly positive genus $g$ yields a 
numerical semigroup $S\cup \{f\}$ of genus $g-1$.

We discuss without proofs in this somewhat
informal and colloquial Section
how to recover the tree-structure from NSG-compositions.
The content of this Section will not be explicitly used in the sequel.
(The tree-structure is however implicitly used when discussing
NSG-compositions with maximal parts of size $3,4$ or $5$.)

The predecessor of a non-trivial NSG-composition $x_1+\cdots +x_{m-1}$ is 
given by $x_1+\cdots +x_{l-1}+(x_l-1)+x_{l+1}+\cdots+x_{m-1}$
if $x_l=\max(x_1,\ldots,x_{m-1})$ is the last maximal part
defining the Frobenius number $f=l+m(x_l-1)$.
A trailing part equal to zero is of course suppressed: 
The predecessor of the composition $g=1+1+\cdots+1+1$ with 
Frobenius number $g$ is given by the composition $1+1+\cdots+1$
of $g-1$.

Children (immediate successors) of the NSG-composition
$1+1+\cdots+1$ (consisting of 
$g$ parts $x_j=1$) are given either by adding an additional part
$x_{g+1}=1$ or by replacing any part $x_j=1$ by $x_j=2$.

Children of a NSG-composition $x_1+\cdots+x_{m-1}$ 
with maximal parts of size
$\max(x_1,\ldots,x_{m-1})\geq 2$ are given as follows:
Let $f=l+m(x_l-1)$ be the Frobenius number associated to 
$x_1+\ldots+x_{m-1}$. Children of $x_1+\cdots+x_{m-1}$ are
given by NSG-compositions
\begin{align}\label{formchild}
&x_1+\cdots+x_{i-1}+(x_i+1)+x_{i+1}+\cdots+x_{m-1}
\end{align}
such that $x_i=x_l$ if $i\leq l$, respectively
$x_i=x_l-1$ if $i>l$. Observe however that compositions
of the form (\ref{formchild}) (with $x_i=x_l$ if $i\leq l$,
respectively, $x_i=x_l-1$ otherwise) do not necessarily satisfy
NSG-inequalities (\ref{fundeqsg}) for indices $s,t$
with $s+t\in\{i,i+m\}$.

The set of all children of a NSG-composition $x_1+\cdots +x_{m-1}$
with maximum at least $2$ corresponds thus to a (perhaps empty)
subset $\mathcal C$ of $\{1,\ldots,m-1\}$ with an element $i$
of $\mathcal C$ defining a child by formula (\ref{formchild}).

Descendants of the NSG-composition $1+1+\dots+1$ with $m=g+1$ are
NSG-compositions of multiplicity at least $m$.

The set of all descendants of a NSG-composition $x_1+\ldots+x_{m-1}$
with $\max(x_1,\ldots,x_{m-1})>1$ can be constructed 
as follows: 
Consider the generalised composition
$\tilde z_1+\cdots+\tilde z_{m-1}$ defined by $\tilde z_i=\infty$ if $i$ belongs to
the set $\mathcal C$ indexing children of $x_1+\cdots+x_{m-1}$
and $\tilde z_i=x_i$ if $i\not\in \mathcal C$.
Applying the NSG-algorithm of Remark \ref{remalgo}
to $\tilde z_1+\cdots +\tilde z_{m-1}$
yields a generalised composition $z_1+\cdots+z_{m-1}$
encoding a smallest semigroup (missing perhaps infinitely
many elements of $\mathbb N$) contained in all
descendants of $x_1+\cdots+x_{m-1}$.
Descendants are encoded by NSG-compositions $y_1+\cdots+y_{m-1}$
with $x_i\leq y_i\leq z_i$. (A composition $y_1+\cdots+y_{m-1}$
with $x_i\leq y_i\leq z_i,\ i=1,\ldots,m-1$ is however not
necessarily a NSG-composition.) Observe that we have $z_i=x_i$ if
$i$ is not in $\mathcal C$. Observe also that the number of
descendants is finite if and only if
$\{1,\ldots,m-1\}\setminus\mathcal C$ generates
$\mathbb Z/m\mathbb Z$.

\subsection{A combinatorial over-tree for successors}

We have seen that descendants of a NSG-composition $x_1+\cdots+
x_{m-1}$ with maximum $\max(x_1,\ldots,x_{m-1})$ at least $2$
correspond to all NSG-compositions $y_1+\cdots+y_{m-1}$
such that $x_i\leq y_i\leq z_i$ where $z_1+\cdots +z_{m-1}$
is a generalised NSG-composition with parts in $\{1,2,\ldots\}\cup\{\infty\}$
defined in terms of $x_1+\cdots+x_{m-1}$. We have $z_i=x_i$ except for $i$
belonging to the set $\mathcal C$ indexing all children of $x_1+\cdots+x_{m-1}$.

We define a sequence
\begin{align}\label{defD}
  \mathcal D&=(z_{i_1}-x_{i_1},z_{i_2}-x_{i_2},\ldots,
              z_{i_k}-x_{i_k})\in \left(\{1,2,\ldots\}\cup\{\infty\}\right)^{\mathcal C}
\end{align}
where the sequence of indices $i_1,\ldots,i_k$ corresponds to all
elements of $\mathcal C$
ordered by $i_a<i_b$ if either $x_{i_a}<x_{i_b}$ or
if $x_{i_a}=x_{i_b}$ and $i_a<i_b$. (The index $i_k$ of the last
element of $\mathcal D$ corresponds thus to the Frobenius number
$f=(x_{i_k}-1)m+i_k$.)

We associate to a sequence $(n_1,\ldots,n_k)\in \left(\{1,2,\ldots,\}\cup
  \{\infty\}\right)^k$ recursively a decorated rooted plane tree
$T(n_1,\ldots,n_k)$ as follows:
The root is decorated by the sequence $(n_1,\ldots,n_k)$.
It has $k$ children defined as the roots of the trees given by
$(n_{i+1},n_{i+2},\ldots,n_k,n_i-1)$ for $i=1,\ldots,k$ with the last
coordinate $n_i-1$ missing if $n_i=1$. Leaves of $T$ are
associated to empty sequences.

A vertex labelled $(5,1,3,1,1,3)$ for example has six children given by $(1,3,1,1,3,4),(3,1,1,3),(1,1,3,2),(1,3),(3)$ and $(2)$.

\begin{prop} The tree $T(n_1,\ldots,n_k)$
  has $\prod_{i=1}^k(n_i+1)$ vertices.
\end{prop}

\begin{proof} The formula holds obviously if $\sum_{i=1}^k n_i=\infty$.
  We can thus assume $n_1,\ldots,n_k\in\{1,2,\ldots\}$.

  The formula holds for the tree $T()$ reduced to its root.

  The induction step reduces to the easy identity
  $$\prod_{i=1}^k(n_i+1)=1+\sum_{i=1}^kn_i\prod_{j=i+1}^k(n_j+1)$$
  with the right hand side obtained as a partial expansion of the product
  $(n_1+1)(n_2+1)\cdots(n_k+1)$.
\end{proof}

A rooted sub-tree $T'$ of a rooted tree $T$ is a sub-tree such that
$v$ in $T'$ for a vertex $v$ implies that the predecessor of $v$ in
$T$ belongs also to $T'$.

\begin{prop} The set of all successors of a NSG-composition
  $x_1+\ldots,x_{m-1}$ with maximum at least $2$ is a rooted
  sub-tree of $T(\mathcal D)$ with $\mathcal D=(z_{i_1}-x_{i_1},z_{i_2}-x_{i_2},\ldots,
z_{i_k}-x_{i_k})$ defined by (\ref{defD}).

\end{prop}

We leave the obvious proof to the reader.

Observe that the set of all successors of a NSG-composition
defines in general a strict sub-tree of $T(\mathcal D)$:
The NSG-composition $3+1+2$ for example gives rise to
$\mathcal D=(\infty,\infty)$.
The corresponding combinatorial tree $T(\mathcal D)$ can be embedded in
$\mathbb R^2$ as follows: Vertices are all elements of $\mathbb N^2$.
Vertices in $\mathbb N\times\{0\}$ are labelled $(\infty,\infty)$.
All other vertices are labelled $(\infty)$. A vertex $(x,0)$
(labelled $(\infty,\infty)$) 
has two successors $(x+1,0)$ (labelled $(\infty,\infty)$) and $(x,1)$
(labelled $(\infty)$). A vertex $(x,y)$ with $y>0$ (labelled $(\infty)$)
has a unique successor $(x,y+1)$ (labelled $(\infty)$).
Only vertices $(x,y)\in\mathbb N^2$ with $y\in\{0,1\}$
correspond to NSG-compositions.

\begin{rem} It is tempting to use the combinatorial trees $T(\mathcal D)$
  for deriving bounds on NSG-compositions. This does
  not seem to pan out: It gives essential only the trivial bounds
  obtained by considering NSG-compositions of genus $g$ as a subset of
  all $2^{g-1}$ compositions of sum $g$.
\end{rem}

\section{Road map for proving Theorem \ref{thmupperboundng}}
\label{sectoutline}

We prove Theorem \ref{thmupperboundng} by exploiting 
the linear structure of compositions.
The proof has two essential
ingredients:

First we show that NSG-compositions with maximal parts
of size at least $4$
have growth-rate strictly smaller than $\omega$.
Asymptotics are thus given by
NSG-compositions with all parts in $\{1,2,3\}$.

In order to bound NSG-compositions with maximal parts larger than $3$,
we chose well-suited pivot-parts $x_p$ of maximal value
in such NSG-compositions.
This cuts a composition $x_1+\cdots+x_{m-1}$ into a left composition
$x_1+\cdots+x_{p-1}$, followed by the pivot-part $x_p$ and
a right composition $x_{p+1}+\cdots+x_{m-1}$.

We construct then a useful upper bound on
the number of possibilities
for left compositions $x_1+\cdots+x_{p-1}$,
retaining only suitably chosen NSG-inequalities given by the first line
of (\ref{fundeqsg}) 
with indices $s,t$ summing up at most to the pivot-index $p$.

Similarly, we construct an upper bound on the number of possibilities
for right compositions $x_{p+1}+\cdots+x_{m-1}$ such that NSG-inequalities
of the second line of (\ref{fundeqsg}) hold for 
suitable choices of $s,t$ such that $m-s,m-t>p$ and $m-s-t\geq p$.
The additional summand $+1$ in the second line of (\ref{fundeqsg}) makes
the study of right compositions a bit spicier.

The product of the two upper bounds for left,
respectively right, compositions is now an
upper bound for the number of all NSG-compositions.
In terms of growth-rates, this translates into the
fact that the maximum of the growth-rates for left,
respectively right, compositions is at least equal to the
growth-rate of all NSG-compositions. This inequality is
sharp for NSG-compositions of maximum $3$, see below.

by taking the maximum among the two upper bounds for left,
respectively right, compositions.

This approach is easy for NSG-compositions with a maximal part
of size at least $6$: It is then enough to work with all inequalities
(\ref{fundeqsg}) such that $s+t=l$, respectively $m-s-t=l$
where $x_l$ is the last part of maximal size defining the Frobenius
number $(x_l-1)m+l$.

For maximum $5$ and $4$, things get more messy: we have
to consider a few additional NSG-inequalities after a careful
choice of pivot-parts. (We need also a technical condition
on such NSG-compositions. NSG-compositions not satisfying the condition
are treated by the procrastinational technique of kicking the can down the road.)

Finally we have to study NSG-compositions with maximum at most $3$.
(The elementary case of maximum $\leq 2$ has already been discussed in
Section \ref{sectmax2}.)
The crucial point for NSG-compositions of maximum $3$ is the
observation that pivoting with respect to the last maximal part
works:
all NSG-inequalities of the second line hold trivially for compositions with
parts in $\{1,2,3\}$. We get thus a ``factorisation''
$$(x_1+\cdots+x_l)+(x_{l+1}+\cdots+x_{m-1})$$
where $x_l=3$ is the last maximal part (defining the Frobenius number
$2m+l$) of such a NSG-composition.
This leads to equation (\ref{formulaforC}) of Theorem \ref{thmupperboundng}.

Finally, we have to show that the generating series $\tilde C$ enumerating
NSG-compositions ending with a last maximal part $x_{m-1}=3$
converges in an open disc of radius strictly larger than $\omega^{-1}$.
The proof is essentially analogous to the proof for an upper bound
on the number of right compositions associated to NSG-compositions
of maximum $4$.

\section{Pivot-factorisation}\label{sectpivot}

A part $x_p$ of a composition $x_1+\cdots+x_{m-1}$ determines
a \emph{pivot-factorisation} given by $x_1+\cdots +x_{p-1}(+x_p)$ and
$(x_p+)+x_{p-1}+\cdots+x_{m-1}$. We call $p$ the \emph{pivot-index}
and $x_p$ the \emph{pivot-part}. We call $x_1+\cdots(+x_p)$ the
\emph{left composition} and $(x_p+)\cdots+x_{m-1}$ the \emph{right composition}
(defined by the pivot-index $p$). The parentheses around $x_p$ indicate
that the inclusion (or exclusion) of $x_p$ is a matter of convention.
We omit generally pivot-parts in left or right compositions.

If $x_1+\cdots+x_{m-1}$ is a NSG-composition, parts of the left composition
$x_1+\cdots(+x_p)$ (with respect to a pivot-part $x_p$) satisfy
the NSG-inequalities
$$x_i+x_j\geq x_{i+j}$$
if $i+j\leq p$ and parts of the right composition satisfy
$$x_i+x_j+1\geq x_{i+j-m}$$
if $i+j\geq p+m$.

Pivot-factorisation with a pivot-part of maximal
size are our main tool for getting useful upper bounds on
NSG-compositions with genus $g$ and maximum $\geq 4$.
More precisely, canonical choices of maximal pivot-parts
in compositions lead to factorisation $LR$
(with $L$, respectively $R$, being generating series accounting for
left, respectively right, compositions) of power
series giving useful upper bounds on NSG-compositions
of certain types.

This idea works nicely for NSG-compositions of maximum
$\mu\geq 6$ and needs a few technical
refinements for $\mu=4$ and $5$.

Pivot-factorisation with respect to a last maximal part with value
$3$ lead to Theorem \ref{thmupperboundng}:
The left factor is given by $1+\tilde C$ (the summand $1$ in
$1+\tilde C(\omega^{-1})$ in (\ref{formulaforC}) accounts for
NSG-compositions of maximum at most $2$),
the right factor, given by the rational series $1/(1-(q+q^2))$, induces the
growth-rate $\omega$.

\section{Generating series, growth-rates}\label{sectgen}

We define the (exponential) \emph{growth-rate} of a sequence of strictly
positive natural integers
$s_1,s_2,\ldots$ by $\gamma=\limsup_{n\rightarrow\infty}\sqrt[n]{s_n}$.
The sequence $s_n$ has exponential growth, if $1<\gamma<\infty$.
We consider henceforth only sequences with
exponential growth.
Given $\epsilon>0$, we have  $s_n<(\gamma+\epsilon)^n$ for almost all
integers $n$ and $s_n>(\gamma-\epsilon)^n$ infinitely often.
The inverse $\rho=1/\gamma$ of the growth-rate $\gamma$ for
$s_0,s_1,\ldots$ is the radius of convergency for the
power series $\sum_{n=0}^\infty s_NR^n$.

\begin{rem} Having exponential growth $\gamma$ is slightly weaker than 
having an asymptotic growth of exponential rate $\gamma$
(defined as $\gamma=\lim_{n\rightarrow\infty}\sqrt[n]{s_n}$). 
\end{rem}

A non-constant power-series with real non-negative coefficients
of growth-rate $\gamma$
defines a holomorphic function in a neighbourhood of $0$ which has
always a smallest singularity at its convergency radius
$\rho=1/\gamma$. If such a series $\sum_{n=0}^\infty s_NP^n$
is rational, then its singularities are isolated and
$\sum_{n=0}^\infty s_Nq_0^n<\infty$ for some strictly positive $q_0$
implies $\gamma<1/q_0$. Strict inequality does however generally
not hold for series which are not rational:
Coefficients of the series
$\sum_{n=0}^\infty\lfloor \gamma^n/(1+n^2)\rfloor q^n$
have growth-rate $\gamma$ for $\gamma>1$
and the series converges for $q$ of absolute value $\vert q\vert=1/\gamma$.
The following result describes however a well-behaved class of
generally irrational power series (with non-negative coefficients):

\begin{lem}\label{lemconvradius}
  Let $S(q)=\sum_{n=n_0}^\infty A_n/(1-B_n)$ be a power-series
with coefficients in $\mathbb N$
  defined by sequences of polynomials $A_n,B_n\in\mathbb N[q]$
  satisfying linear recursions with coefficients in $\mathbb Q[q]$.

  Suppose that there exists a strictly positive real number $\rho_0$
  such that $S(q)$ converges for $q=\rho_0$ and such that the evaluations
  of $A_n,B_n$ at $q=\rho_0$ decay exponentially fast. Then $S(q)$
  has convergency radius strictly larger than $\rho_0$.
\end{lem}

\begin{proof}[Proof of Lemma \ref{lemconvradius}]
  The hypotheses imply that evaluations at $\rho_0$ of the characteristic
  polynomials defining linear recursions for $A_n$ and $B_n$ have
  all roots in the open complex unit disc.
  This condition (which implies exponentially fast decay)
  holds by continuity for $q$ close enough to $\rho_0$.
  
  A similar arguments shows that all evaluations of $B_n$ at $q$ are
  strictly smaller than $1$ for $q$ close enough to $\rho_0$.

  Non-negativity of all involved coefficients shows now that
  the convergency radius of $S$ is strictly larger than $\rho_0$.
\end{proof}  

\begin{rem} Lemma \ref{lemconvradius} streamlines a few proofs. It can
  however easily be replaced by a few computations involving
  eigenvalues and eigenvectors of transfer matrices.
\end{rem}

Throughout the paper we will also use several times the trivial
fact that the convergency radius of a finite product of
power-series is at least equal to the minimal convergency radius
among factors.

\section{Weakly admissible compositions}\label{sectweakly}

A composition $x_1+\cdots+x_{m-1}$ with last maximal part $x_l=\max(x_1,\ldots,
x_l)>\max(x_{l+1},\ldots,x_{m-1})$ is \emph{weakly admissible}
if $x_l\leq \min(x_1+x_{l-1},x_2+x_{l-2},\ldots,x_{l-1}+x_1)$
and $x_l\leq 1+\min(x_{l+1}+x_{m-1},x_{l+2}+x_{m-2},\ldots,x_{m-1}+x_{l+1})$.
Otherwise stated,
we require only
NSG-inequalities (\ref{fundeqsg}) 
involving the last maximal part $x_l$ of $x_1+\cdots+x_{m-1}$.

Since NSG-compositions are weakly admissible we get crude upper
bounds for NSG-compositions by counting weakly admissible
compositions.

Observe however that the composition $1+3+3$ for example
is weakly admissible but is not a NSG-composition.

This section is devoted to generating series of weakly admissible 
compositions having maximal parts of given size $\mu$.

Weakly admissible compositions of maximum at least $6$ have growth-rate
strictly smaller than $\omega$ and give thus
useful upper bounds
for NSG-compositions with maximal parts of size at least $6$.

Refinements are needed for useful bounds on 
NSG-compositions with maximal parts of size $4$ and $5$.
NSG-compositions with maximal parts of size at most $3$
are asymptotically generic (their proportion among all
NSG-compositions tends to $1$ for $g\rightarrow\infty$)
and induce the growth-rate for numerical semigroups.

For $n\geq 1$, we consider the generating polynomial
\begin{align}\label{defIn}
I_n=\sum_{1\leq a,b\leq n\leq a+b}q^{a+b}
\end{align}
of all compositions $a+b$ with total sum at least $n$ into two parts
$a,b$ not exceeding $n$.

The polynomial $I_n$ can be computed by removing 
all contributions of degree strictly less than $n$ (corresponding 
to compositions with two parts summing up to integers strictly
smaller than $n$) from $(q+q^2+\ldots+q^n)^2$.
This implies easily the closed formula
\begin{align}\label{formulaIn}
I_n&=q^n\left(-2+\sum_{i=0}^n(n+1-i)q^i\right).
\end{align}

The following table gives the first few polynomials $I_n$ 
and their evaluations
$I_n(\omega^{-1})$ in $\mathbb Q[\sqrt{5}]$, together with a decimal 
approximation, at the inverse $\omega^{-1}$ of the golden
number $\omega=\frac{1+\sqrt{5}}{2}$:
$$\begin{array}{l|c|c|c}
k&I_k&I_k(1/\omega)&\sim\\
\hline
1&q^2&(3-\sqrt{5})/2&0.3820\\
2&q^2+2q^3+q^4&1&1\\
3&2q^3+3q^4+2q^5+q^6&(9-3\sqrt{5})/2&1.1459\\
4&3q^4+4q^5+3q^6+2q^7+q^8&10-4\sqrt{5}&1.0557\\
5&4q^5+5q^6+4q^7+3q^8+2q^9+q^{10}&21-9\sqrt{5}&0.8754\\
6&5q^6+6q^7+5q^8+4q^9+3q^{10}+2q^{11}+q^{12}&70-31\sqrt{5}&0.6819
\end{array}$$

\begin{prop}\label{propweaklyadmgenser} The generating series
$W_\mu(q)$ for weakly admissible compositions $x_1+\cdots+x_{m-1}$
of genus $g$ with maximal
parts $\max(x_1,\dots,x_{m-1})=\mu\geq 2$ is given by
\begin{align}\label{formulaweaklyadma}
W_\mu(q)&=
\frac{1+\sum_{i=\lceil \mu/2\rceil}^\mu q^i}{1-I_\mu}q^\mu
\frac{1+\sum_{i=\lfloor \mu/2\rfloor }^{\mu-1}q^i}{1-I_{\mu-1}}.
\end{align}
\end{prop}

\begin{rem}
  Since all compositions with parts in $\{1,2\}$ are
  NSG-compositions, Propositions \ref{propparts12}
  and \ref{propweaklyadmgenser}
  imply the identity
  $$\frac{1}{1-(q+q^2)}=\frac{1}{1-q}+W_2(q)\ .$$
\end{rem}

\begin{proof}[Proof of Proposition \ref{propweaklyadmgenser}]
Let $x_1+\cdots+x_{l-1}+x_l+x_{l+1}+\cdots+x_{m-1}$ 
be a weakly admissible 
composition with $x_1,\ldots,x_{l-1}\leq x_l=\mu>x_{l+1},\ldots,x_{m-1}$.
We have thus $1\leq x_i,x_{l-i}\leq \mu\leq x_i+x_{l-i}$
for $i<l/2$. For odd $l$ there are $I_\mu^{(l-1)/2}$ possibilities
satisfying these inequalities. For $l$ even, we have moreover to
choose a coefficient $x_{l/2}$ in $\{\lceil \mu/2\rceil,\ldots,\mu\}$.
Summing over $l$ in $\mathbb N\setminus\{0\}$ we get the left factor 
$(1+\sum_{i=\lceil \mu/2\rceil}^\mu q^j)/(1-I_\mu)$ of (\ref{formulaweaklyadma})
enumerating all possibilities for left compositions with respect to the
pivot-part $x_l=\mu$ given by the last maximal part $x_l$.

The central factor $q^\mu$ accounts for the pivot-part 
$x_l=\mu$.

The final right factor corresponds to all possibilities
involving the parts $x_{l+1},\ldots,x_{m-1}\in\{1,\ldots,\mu-1\}$
of right compositions
following the pivot-part $x_l=\mu$: 
We have $1\leq x_{l+i},x_{m-i}\leq \mu-1\leq
x_{l+i}+x_{m-i}$. Such pairs $(x_{l+i},x_{m-i})$
are thus encoded by powers of 
$I_{\mu-1}$ and we have moreover a choice of $x_{(m+l)/2}$
in $\{\lfloor \mu/2\rfloor,\ldots,\mu-1\}$ if $m+l$ is even.
\end{proof}

\begin{rem} Formula (\ref{formulaweaklyadma}) gives crude upper bounds:
It counts only compositions satisfying NSG-inequalities (\ref{fundeqsg})
involving the last maximal part. Only a small proportion of weakly
admissible compositions with maximum strictly larger than $2$
satisfy all NSG-inequalities.
\end{rem}

\section{NSG-compositions with maximum $\geq 6$}\label{sectmax6}

Proposition \ref{propweaklyadmgenser} gives useful
upper bounds for the number of NSG-compositions with maximum
at least $6$, as suggested by the evaluations $I_n(\omega^{-1})$
of the first few polynomials $I_n$ defined by (\ref{formulaIn}):

\begin{prop}\label{propmaxgeq6} Numbers of
  weakly admissible compositions with maximum at least $6$
  have growth-rate strictly smaller than $\omega$.
\end{prop}

\begin{proof} Proposition \ref{propweaklyadmgenser} shows that it 
  is enough to prove that $\sum_{n=6}^\infty W_n(q)$ (for $W_n(q)$
  defined by (\ref{formulaweaklyadma})) converges in
an open disc of radius strictly larger than $\omega^{-1}$.

This holds clearly for $W_6(q)$ which converges in the open disc
of radius the strictly positive root $0.6318\ldots>\omega^{-1}$
of $1-I_5=1-4q^5-5q^6-4q^7-3q^8-2q^9-q^{10}$.

Formula (\ref{formulaIn}) shows that coefficients of the rational fraction
$$\mathbf I_6=\sum_{j=6}^\infty (j-1)q^j=q^6\frac{5-4q}{(1-q)^2}$$
yield upper bounds on the coefficients of $I_n$ for $n\geq 6$.

The convergency radius of $\sum_{n=7}^\infty W_n(q)$
is thus at least as large as the convergency radius 
of the rational fraction
$$\sum_{n=7}^\infty\frac{\sum_{j=0}^\infty q^j}{1-\mathbf I_6}q^n\frac{\sum_{j=0}^\infty
q^j}{1-\mathbf I_6}=\frac{q^7}{(1-q)^3(1-\mathbf I_6)^2}$$
given by the positive root $0.6206\ldots>\omega^{-1}$ of 
the polynomial $(1-q)^2(1-\mathbf I_6)=1-2q+q^2-5q^6+4q^7$.
\end{proof}

\section{NSG-compositions with maximum $5$}\label{sectmax5}

The rational generating series
$W_5(q)$, given by Proposition \ref{propweaklyadmgenser} and
enumerating weakly admissible compositions of maximum $5$, 
involves $1-I_4(q)$ (accounting for
right compositions) in its denominator. Since $1-I_4$
has a root in $(0,\omega^{-1})$, the growth-rate of
weakly admissible compositions
of maximum $5$ exceeds $\omega$.
Obtaining useful upper bounds for NSG-compositions of maximum $5$
(trickier than obtaining the corresponding
results for NSG-compositions of maximum at least $6$, see
Section \ref{sectmax6}) requires thus
additional NSG-inequalities (\ref{fundeqsg}).

\begin{prop}\label{propmax5} There exists a strictly positive constant
  $\kappa_5<\omega$ such that NSG-compositions with maximum $5$ have
  growth-rate at most equal to $\max(\gamma_4,\kappa_5)$
  where $\gamma_4$ is the growth-rate of NSG-compositions with maximum $4$.
\end{prop}

Proposition \ref{propmax5} follows trivially from the
two following results:

\begin{prop}\label{propuniquemaximum5} The growth-rate of 
NSG-compositions with maximum $5$ having a unique maximal part is at
most equal to the growth-rate $\gamma_4$ of NSG-compositions
with maximum $4$.
\end{prop}

\begin{prop}\label{propmax5twoparts} The growth rate of
  NSG-compositions having at least two maximal parts of size $5$
  is strictly smaller than $\omega$.
\end{prop}

\begin{rem} We will show later that the constant $\gamma_4$ of
  Proposition \ref{propmax5} satisfies the inequality
  $\gamma_4<\omega$ (see Proposition \ref{propmax4} and
  Theorem \ref{thmradiustildeC}).
\end{rem}

\subsection{Proof of Proposition \ref{propuniquemaximum5}}

\begin{proof}[Proof of Proposition \ref{propuniquemaximum5}] 
A unique maximal part $x_l=5$ of size
$5$ corresponds to the Frobenius number $f=4m+l$
of such a NSG-composition $x_1+\cdots+x_{m-1}$.
Adding the Frobenius element $f$ to the associated numerical
semigroup amounts to replacing
$x_l=5$ by $x_l=4$ and results in a NSG-composition
with maximum $4$ (and genus decreased by $1$).
Such reductions yield any given NSG-composition
of genus $g$ and maximum $4$ less than $g$ times
(since $m\leq g+1$ with equality only for $1+1+\ldots+1$).
Numbers of NSG-compositions with a unique maximal
part of size $5$ are thus bounded by coefficients of $q^2G_4'$ where 
$G_4$ is the generating series for all NSG-compositions with maximum $4$.
The result follows by observing that coefficients of $G_4$ and
of its derivative $G'_4$ have identical growth-rates.
\end{proof}

\subsection{Proof of Proposition \ref{propmax5twoparts}}

\begin{prop}\label{propmax5largeFrob} NSG-compositions
  $x_1+\cdots+x_{m-1}$ with a last maximal part $x_l=5$
  such that $3l\geq m-1$ have growth rate at most equal to $1/\rho<\omega$
  where $\rho=0.6189\ldots>\omega^{-1}$
  is the positive root of $1-I_5I_4^2$ for
  $I_n$ given by (\ref{formulaIn}).
\end{prop}

\begin{proof}[Proof of Proposition \ref{propmax5largeFrob}]
  Choosing the index $l$ of the last maximal part $x_l=5$
  in such a NSG-composition as a pivot,
the proof of Proposition \ref{propweaklyadmgenser} shows that the
number of such NSG-compositions of genus $g$ (with multiplicity $m$ and
last maximal part $x_l=5$) is 
bounded by the coefficient of $q^g$ in
$$(1+q^3+q^4+q^5)(1+q^2+q^3+q^4)
I_5^{\lfloor (l-1)/2)\rfloor}I_4^{\lfloor (m-1-l)/2\rfloor}$$
where
\begin{align*}
I_4&=3q^4+4q^5+3q^6+2q^7+q^8,\\
I_5&=4q^5+5q^6+4q^7+3q^8+2q^9+q^{10}
\end{align*}
(see Formula (\ref{formulaIn})).

The inequality $3l\geq m-1$ implies that $(m-1-l)/2-1$ is at most 
twice as large as $(l-1)/2$.

Neglecting polynomial factors, we have reduced the proof of
Proposition \ref{propuniquemaximum5} to the study of the 
convergency radius of
\begin{align}\label{formcvgncy5}
  &\sum_{a=0}^\infty \sum_{b=0}^{2a}I_5^aI_4^b\ .
\end{align}

Rewriting (\ref{formcvgncy5}) by regrouping $I_5I_4^2$
we can work with
$$\frac{1}{(1-I_5I_4^2)(1-I_5)}$$ 
(up to neglecting a polynomial factor)
which converges on the open disc of radius $\rho>\omega^{-1}$
the strictly positive root $0.6189\ldots$ of $1-I_5I_4^2$.
\end{proof}

\begin{proof}[Proof of Proposition \ref{propmax5twoparts}]
  Given a NSG-composition $x_1+\cdots+x_{m-1}$ having at least
  two maximal parts of size $5$
let $k$ and $l$ with $1\leq k<l<m$ be the indices of the two last maximal
parts $x_k=x_l=5$.

The result holds by Proposition \ref{propmax5largeFrob} for
all NSG-compositions such that $3l\geq m-1$.

We are now left with the case of NSG-compositions $x_1+\cdots+x_{m-1}$
with indices of the two last maximal parts satisfying
$k<l<(m-1)/3$.

We give again a factorised upper bound for all such NSG-compositions.

We use $\frac{1+q^3+q^4+q^5}{1-I_5}$ (see the proof of Proposition
\ref{propweaklyadmgenser}) for upper bounds on numbers of left
compositions with respect to the pivot-part $x_k$ given by
the second-last maximum. (This works since coefficients of
$1/(1-I_5)$ have growth-rate strictly smaller than $\omega$.)

In order to get upper bounds for right compositions, we introduce
a graph $\Gamma$ with vertices $k+1,k+2,\ldots,m-1$ and edges
$\{i,j\}$ if $i+j-m\in\{k,l\}$.
The graph $\Gamma$ is a union of paths (trees with at most two leaves
and interior vertices of degree $2$). A connected component of $\Gamma$ is
\emph{ordinary} if it contains no endpoint in $\{(k+m)/2,(l+m)/2\}
\cap \mathbb N$. It is \emph{exceptional} otherwise. Ordinary connected
components of $\Gamma$ have an even number of vertices
and contain a central edge.

\begin{figure}[h]\label{figdihex}
\epsfysize=10cm
\center{\epsfbox{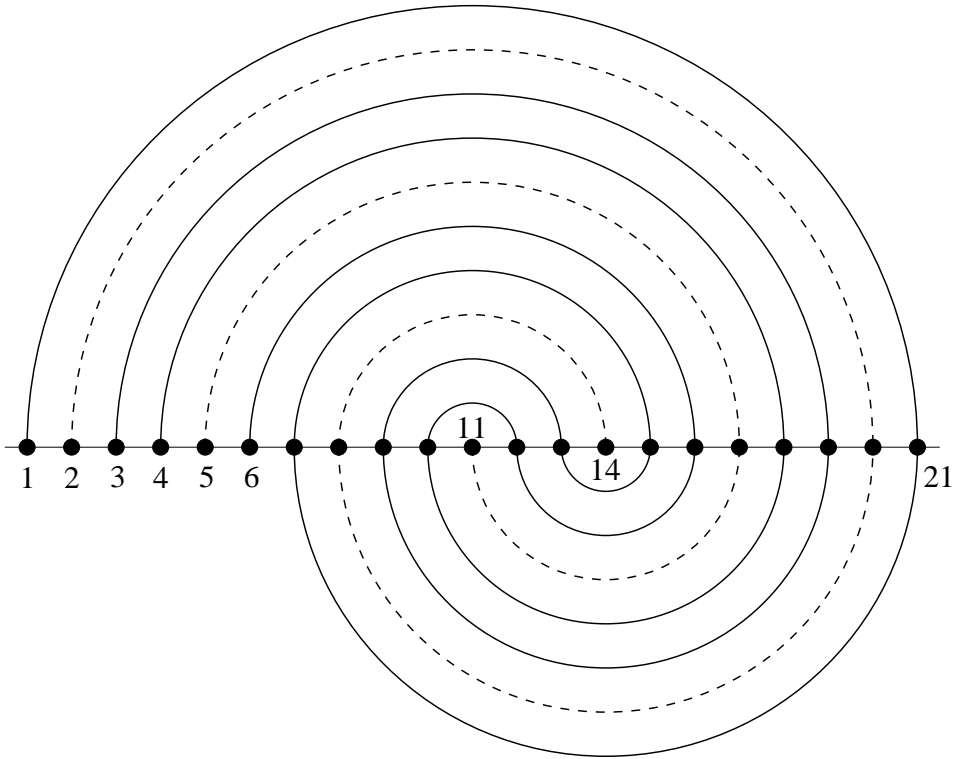}}
\caption{An example of a graph $\Gamma$.}
\end{figure}

Figure 1
shows an example with vertices $k+1,k+2,\ldots$
simply numbered $1,2$, where where $m=k+22$ and $l=k+6$. Edges $\{i,j\}$ with
$i+j=k+m$, respectively $i+j=l+m$,
are represented by upper, respectively lower, half-circles.
Edges of exceptional
components are dashed. (The horizontal coordinate-axis is not part of
$\Gamma$.)

The first $l-k$ vertices $k+1,k+2,\ldots,l$ of $\Gamma$ are leaves of
all connected components of $\Gamma$.

Ordinary connected components have length $2n\pm 1$ for some integer $n$
(depending on $k,l,m$). Exceptional components have length $n$
or $n-1$.

The construction of $\Gamma$ implies
$x_i+x_j\geq 4$ if $\{i,j\}$ is an edge of $\Gamma$
associated to a NSG-composition $x_1+\ldots+x_{m-1}$ as above.
Generating functions of parts supported by connected components
(maximal paths) of $\Gamma$ can be considered as partition-functions
of constrained spin models with spins $\{1,2,3,4\}$
such that all pairs of adjacent spins have sum at least $4$.

We can compute the generating series (partition-functions) of
this model on paths by an easy application of the transfer-matrix method
originating in statistical physics
and widely used in enumerative combinatorics,
see e.g. Chapter 4.7 of \cite{StI}. (Details are easy and
can be understood without prior knowledge of spin models.)

We denote by $A_n,B_n,C_n$ the generating series (with respect
to the weight $q^{s_0+s_1+\ldots+s_{n}}$) for sequences $(s_0,s_1,\ldots,
s_n)\in\{1,2,3,4\}^{n+1}$ of length
$n+1$ with coefficients in $\{1,2,3,4\}$ such that $s_i+s_{i+1}\geq 4$
for $i=0,\ldots,n-1$ and ending with $s_n=1$ for elements counted by $A_n$,
$s_n=2$ for $B_n$ and $s_n\in \{3,4\}$ for $C_n$.
We have $A_0=q,B_0=q^2,C_0=q^3+q^4$ and
\begin{align}\label{formularecABC5}
\left(\begin{array}{c}A_n\\B_n\\C_n\end{array}\right)&=
\left(\begin{array}{ccc}0&0&q\\0&q^2&q^2\\q^3+q^4&q^3+q^4&q^3+q^4
\end{array}\right)^n\left(\begin{array}{c}A_0\\B_0\\C_0
\end{array}\right)\ ,
\end{align}
by an easy induction.
The transfer-matrix involved in (\ref{formularecABC5})
has eigenvalues $-0.3981,0.2476,0.9143$ at $q=\omega^{-1}$.
This shows that the evaluations $P_n(\omega^{-1})$ of
the polynomials $P_n=A_n+B_n+C_n$ decay
exponentially fast to $0$. The evaluations $P_n(\omega^{-1})$
are strictly smaller than $1$ for $n\geq 3$.

Lemma \ref{lemconvradius} implies that the series
\begin{align}\label{seriesmax5}
  &\left(1+\sum_{n=0}^\infty P_n\right)^2q
   \left(1+\sum_{n=3}^\infty
   \frac{P_{2n-1}+P_{2n+1}}{(1-P_{2n-1})(1-P_{2n+1})}\right)
\end{align}
converges on an open disc of radius strictly larger than $\omega^{-1}$.

Coefficients of (\ref{seriesmax5}) are upper bounds on the number
of possibilities for right compositions of NSG-compositions
with pivot-part the second last maximum $x_k=5$:
The squared first factor has convergency radius strictly larger than
$\omega^{-1}$ by Lemma \ref{lemconvradius}
and accounts for perhaps existing exceptional
components, the isolated factor $q$ corrects for the fact that
$x_l=5\not\in\{1,2,3,4\}$. The final factor (which converges
for $q$ slightly larger than $\omega^{-1}$, see Lemma \ref{lemconvradius})
yields a coarse upper
bound for contributions coming from ordinary components of $\Gamma$.
The summand $1$ of the final factor
is needed for NSG-compositions with graphs reduced to
exceptional components. The sum of the final factor starts at $n=3$ since the
inequality $l<(m-1)/3$ ensures length at least $5$ for
ordinary components of $\Gamma$.
The numerator $P_{2n-1}+P_{2n+1}$ (necessary for convergency)
is due to the assumed existence of at least one ordinary component
of length $2n-1$ or $2n+1$.
(NSG-compositions without ordinary components in the associated graph
$\Gamma$ are accounted for by the
summand $1$ of the last factor, see above).
\end{proof}

\section{NSG-compositions with maximum $4$}\label{sectmax4}

This section contains a proof of the following result:

\begin{prop}\label{propmax4} We assume the inequality
  $\tilde \gamma<\omega$ for the growth-rate of
  the generating series $\tilde C(q)$
  enumerating NSG-compositions of maximum 3 ending with a
  maximal part.

  There exists then a positive constant
  $\kappa_4$ strictly smaller than $\omega$ such that
  NSG-compositions of maximum $4$ have growth-rate at most equal
  to 
$\max(\kappa_4,1/\rho,\sqrt[3]{\tilde\gamma\omega^2})$
  where $\rho=0.71667\ldots>\omega^{-1}$ is the positive root of
  $1-(2q^3+q^4)$. 
\end{prop}

The inequality $\tilde\gamma<\omega$ will be proven in
Theorem \ref{thmradiustildeC} whose proof uses parts (contained in
Sections \ref{sectCayley}-\ref{subsectepsilon4} and independent
of the assumption $\tilde \gamma<\omega$) of the proof
of Proposition \ref{propmax4}.

We start by outlining the proof of Proposition \ref{propmax4},
hopefully providing the reader with
a ``magic flute'' guiding him through the foggier parts.
The main difficulty is due to the fact that we have to work
with NSG-inequalities involving three maximal parts when
considering right compositions with respect to a suitable
maximal pivot-part.

The proof remains however roughly similar to the proof of
the corresponding result (given by  Proposition \ref{propmax5})
for NSG-compositions of maximum $5$.

We start by studying the growth-rate of NSG-compositions with maximum 4
having a bounded number of maximal parts. This is done procrastinationally
in Section \ref{subsectA} whose main result, Proposition
\ref{propmax4Aparts} is the analogue of
Proposition \ref{propuniquemaximum5}.

Section \ref{subsectleftcomp4} is devoted to left compositions with
respect to an arbitrary maximal pivot-part.

Section \ref{sectFrobclose4m} is the analogue of Proposition
\ref{propmax5largeFrob}: It deals with NSG-compositions whose
Frobenius numbers are close to $4m$.

Section \ref{sectCayley} discusses Cayley and Schreier graphs 
of (some) groups.

Section \ref{sectlane} discusses spin models of lanes,
used in the proof of Section \ref{subsectepsilon4}

Section \ref{subsectepsilon4} contains the core of the proof: It
introduces the $\epsilon$-condition
and bounds the growth rate of right compositions satisfying the 
$\epsilon$-condition.

The content of Sections \ref{sectFrobclose4m}--\ref{subsectepsilon4}
corresponds more or less to the statement and proof of
Proposition \ref{propmax5twoparts} in the case of NSG-compositions of
maximum $5$.

Section \ref{sectproofpm4} ties up loose ends.


\subsection{NSG-compositions with a bounded number of
  maximal parts of size $4$}\label{subsectA}

\begin{prop}\label{propmax4Aparts}
Given a natural integer $A$, the growth-rate of
  NSG-compositions with maximum $4$ having at most $A$ maximal parts
  is at most equal to $\max(1/\rho,\sqrt[3]{\tilde\gamma\omega^2})$
  where $\rho=0.71667\ldots>\omega^{-1}$ is the positive root of
  $1-(2q^3+q^4)$ and where $\tilde \gamma<\omega$ is the growth-rate of
  the generating series $\tilde C(q)$
  enumerating NSG-compositions with maximum $3$ ending with a part
  of maximal size.
\end{prop}

A straightforward modification of the proof of Proposition
\ref{propuniquemaximum5} does unfortunately not work: NSG-compositions
with a maximal part of size $3$ are generic and have growth-rate
$\omega$ due to the possibility of arbitrary long 'tails' involving only
summands $1$ and $2$. In order to circumvent this difficulty,
we discuss first NSG-compositions with 'short tails'.

A NSG-composition $x_1+\ldots+x_{m-1}$ is \emph{short-tailed} if it involves
a summand $x_l\geq 3$ with index $l\geq m/2$.

\begin{lem}\label{lemsht} The generating series $S$ for short-tailed
  NSG-compositions with maximum $3$
  has growth rate at most $\sqrt[3]{\tilde\gamma\omega^2}$
  where $\tilde \gamma<\omega$ is the growth-rate of
  the generating series $\tilde C(q)$
  enumerating NSG-compositions with maximum $3$ ending with a part
  of maximal size.
\end{lem}

\begin{proof}[Proof of Lemma \ref{lemsht}]
Let $\mathbf x=x_1+\cdots+x_l+\cdots +x_{m-1}$ be a short-tailed NSG-composition
with last maximum $x_l=3$ for $2l\geq m$.
Since $x_1,\ldots,x_l\geq 1$ and $x_{l+1},\ldots,x_{m-1}\leq 2$ with $l>m-1-l$
we have $2(x_1+\cdots+x_l)=2l>2(m-1-l)\geq x_{l+1}+\cdots+x_{m-1}$.
In particular, the NSG-composition $x_1+\cdots+x_l$
contributes at least $g/3$ to the genus $g=x_1+\cdots+x_{m-1}$
of $\mathbf x$.

Denoting by $\tilde c_g$ the number of NSG-compositions of genus $g$
ending with a final maximal part of size $3$ (with growth-rate $\tilde \gamma$,
involved in the
generating series $\tilde C=\sum_{g=3}^\infty \tilde c_gq^g$, see
(\ref{formulafortildeC})),
the series
\begin{align}\label{uppshtld}
U&=\sum_{g=3}^\infty \left(\sum_{k=\lceil g/3\rceil}^g\tilde c_kF_{g-k}\right)
                                q^g
\end{align}
(where the Fibonacci numbers $F_{g-k}$ count possibilities for the final
NSG-composition $x_{l+1}+\cdots+x_{m-1}$ with all parts in $\{1,2\}$,
see Proposition \ref{propparts12})
with growth rate $\sqrt[3]{\tilde\gamma\omega^2}$
gives upper bounds on the coefficients of $S$ counting short-tailed
NSG-compositions with maximum $3$.
\end{proof}

\begin{proof}[Proof of Proposition \ref{propmax4Aparts}]
  Replacing all parts of size $4$ in such a NSG-composition
$\mathbf x=x_1+\cdots+x_{m-1}$ with parts of size $3$, we get a NSG-composition $\underline{
  \mathbf{x}}=\underline{x}_1+\cdots +\underline{x}_{m-1}$
with parts $\underline{x}_i=\min(x_i,3)$
of maximal size $3$.

We count first all such NSG-compositions $\mathbf x$ (having at most
$A$ maximal parts of size $4$) such that
$\underline{\mathbf{x}}$ is short-tailed. Since
such a short-tailed NSG-composition of genus $g$ can be lifted
into at most ${g\choose a}$ NSG-compositions $\mathbf x$ having $a$
maximal parts of size $4$, we can bound the number of such NSG-compositions
$\mathbf x$ by applying the same trick used with $a=1$ in the proof of
Proposition \ref{propuniquemaximum5}. More precisely, numbers of such
NSG-compositions $\mathbf x$ (with maximum $4$ arising at most $A$ times,
associated to a short-tailed NSG-composition $\underline{\mathbf{x}}$
as above) are bounded above by coefficients of
\begin{align}\label{defUA}
U_A&=\sum_{a=1}^Aq^{2a}\frac{d^a}{dq^a}U
     =\sum_{a=1}^Aq^{2a}\frac{d^a}{dq^a}\left(\sum_{g=3}^\infty \left(\sum_{k=\lceil g/3\rceil}^g\tilde c_kF_{g-k}\right)q^g\right)
\end{align}
with $U$ the series of upper bounds for short-tailed NSG-compositions
of maximum $3$ given by (\ref{uppshtld}). Since finite sums of
derivatives do not increase growth-rates, the growth rate of the
series $U_A$ is still given by $\sqrt[3]{\tilde\gamma\omega^2}$.

We consider now NSG-compositions $\mathbf x$ having at most $A$ maximal
parts of size $4$ giving rise to NSG-compositions
$\underline{\mathbf{x}}$ which are not short-tailed.
Such a NSG-composition $\mathbf x=x_1+\cdots+x_l+x_{l+1}+\cdots+x_{m-1}$
has thus a last part $x_l$ of size $3$ or $4$ indexed by $l<m/2$.
Since it has maximum $4$, it contains a part $x_k$ (not necessarily
distinct from $x_l$) of size $4$ with $k\leq l$.
Since $k\in\{1,\ldots,l\}$ with $l<m/2$,
the two (not necessarily distinct) indices $l+1$ and $m+k-l-1$
belong to $\{l+1,\ldots,m-1\}$ and the NSG-inequalities
(\ref{fundeqsg}) involving $x_k=4$ yield
$x_{l+i}+x_{m+k-l-i}\geq 3$ with $x_{l+i},x_{m+k-l-i}$ in $\{1,2\}$
for $i\geq 1$ such that $l+i\leq m+k-l-i$.
The contribution of all such (not necessarily distinct pairs)
to the generating series involved in Proposition \ref{propmax4Aparts}
can be bounded above by
$\frac{1+q^2}{1-(2q^3+q^4)}$. Removing all summands $x_{l+1},x_{l+2},\ldots,
x_{m+k-l-1}$ involved in such pairs from $\mathbf x$ and replacing
maximal parts of size $4$ by parts of size $3$,
we get a NSG-composition
$$\tilde{\underline{\mathbf{x}}}=\min(x_1,3)+\cdots+\min(x_l,3)+
x_{m+k-l}+x_{m+k-l+1}+\cdots+x_{m-1}$$
with tail-length (given by the number of final parts in $\{1,2\}$
following $\underline{x}_l=3$)
$$m-1-(m+k-l-1)=l-k<l\ .$$
This implies that $\tilde{\underline{\mathbf{x}}}$ is short-tailed.

The generating series of NSG-compositions having at most $A$ parts of
maximal size $4$ can thus be bounded above by the generating series
$$U_A\frac{1+q^2}{1-(2q^3+q^4)}$$
(the constant coefficient $1$ of $\frac{1+q^2}{1-(2q^3+q^4)}$
accounts for NSG-compositions discussed at the start of the proof)
with growth-rate given by Proposition \ref{propmax4Aparts}.
\end{proof}

\subsection{Left compositions}\label{subsectleftcomp4}

We set
\begin{align}\label{formulatildePn4}
\tilde P_n&=\left(\begin{array}{ccc}1&1&1\end{array}\right)
\left(\begin{array}{ccc}0&0&q\\0&q^2&q^2\\q^3+q^4&q^3+q^4&q^3+q^4
      \end{array}\right)^n
                                                           \left(\begin{array}{c}q\\q^2\\q^3\end{array}\right)
\end{align}
for $n\geq 1$.

\begin{prop}\label{propleftcomp4} The coefficients
  of the generating series for distinct
  left compositions with respect to a
  maximal pivot-part $4$ are bounded above by coefficients of
  \begin{align}\label{formulaleft4}
    (1+q)\left(1+q^2+q^3+q^4+\sum_{n=1}^\infty\tilde P_n\right)^2
    \left(1+\sum_{n=1}^\infty \frac{\tilde P_{2n-1}+\tilde P_{2n+1}}{(1-\tilde P_{2n-1})(1-\tilde P_{2n+1})}\right)
    \end{align}
  which converges on an open disc of radius strictly larger than
  $\omega^{-1}$.
\end{prop}

\begin{proof}
In order to study the generating series associated to left
compositions $x_1+\cdots+x_{l-1}(+x_l)$ with respect to a maximal
pivot-part $x_l=4$, we consider the graph $\Gamma$ with vertices
$1,\ldots,l-1$ and edges $\{i,j\}$ if $i+j\in\{k,l\}$
where $k<l\leq 2k$ is maximal such that $x_k=4$ if such an index
$k\geq l/2$ exists.
The graph $\Gamma$ is a union of disjoint edges $\{i,l-i\}$
(and of the isolated
vertex $l/2$ if $l$ is even) if the index $k$ does not exist.
Connected components of
$\Gamma$ are paths having at least one end-point in
$k,k+1,\ldots,l-1$ (respectively $\lceil l/2\rceil,
\ldots,l-1$ if $k$ does not exist). (More precisely, they all start
and end at points in $k,\ldots,l-1$, except for at most two
paths starting at elements of $k,\ldots,l-1$ and ending
in $\{k/2,l/2\}\cap \mathbb N$.)

Parts (spins) associated to initial vertices in $\{k,k+1,\ldots,l-1\}$
are elements of $\{1,2,3\}$,
except for the path starting at $x_k=4$ if $k$ exists. All other spins
are in $\{1,2,3,4\}$ and we have $x_i+x_j\geq 4$
for edges $\{i,j\}$ of $\Gamma$.

The generating series of a path of length
$n\geq 1$ not starting at $x_k$
is thus given by the polynomial $\tilde P_n$ (obtained
by the transfer matrix method, compare with formula \ref{formularecABC5})
defined by formula (\ref{formulatildePn4}). 

Formula (\ref{formulaleft4}) is analogous to Formula (\ref{seriesmax5}):
The factor $(1+q)$ of formula (\ref{formulaleft4})
accounts for the correction $x_k=4\not\in\{1,2,3\}$
of the (not necessarily existing) path starting at $k$ with initial spin
$x_k=4$.
The squared factor takes into account the possible existence of exceptional
paths ending at $\{l/2,k/2\}\cap\mathbb N$. The final factor deals with
all ordinary paths (not containing a vertex of $\{l/2,k/2\}\cap \mathbb N$)
of $\Gamma$.

Convergency of the series (\ref{formulaleft4})
for $q$ slightly larger than $\omega^{-1}$
follows from Lemma \ref{lemconvradius}: Elementary (and somewhat lengthy)
computations involving the transfer matrix of (\ref{formulatildePn4})
show that evaluations at $q=\omega^{-1}$
of $\tilde P_n$ decay exponentially fast and are strictly smaller than $1$
for $n\geq 1$.
\end{proof}

\subsection{Frobenius numbers close to $4m$}\label{sectFrobclose4m}

\begin{prop}\label{propmax4delta}
  There exists $\delta>0$ such that
  NSG-compositions $x_1+\cdots+x_l+\cdots+x_{m-1}$
  satisfying the inequality $(m-1-l)\leq \delta(m-1)$ for the
  index $l$ of their last maximal part $x_l=4$
  have growth-rate strictly smaller than $\omega$.
\end{prop}

\begin{rem} The Frobenius number $3m+l$ of NSG-compositions
described by Proposition \ref{propmax4delta}
is close to $4m$ if $\delta$ is small.
\end{rem}

\begin{proof}[Proof of Proposition \ref{propmax4delta}]
  We consider the pivot-factorisation of such a NSG-composition
  $x_1+\cdots+x_{m-1}$ with pivot-part the last maximal part $x_l$.
  Using the NSG-inequalities $x_l=4\leq 1+x_{l+i}+x_{m-i}$
  we see that the right composition $x_{l+1}+\cdots+x_{m-1}$
  (with $m-1-l$ parts in $\{1,2,3\}$) contributes at most a factor
  of $\sum_{n=0}^{3(m-1-l)}r_nq^n$ to the generating series where 
  $\sum_{n=1}^\infty r_nq^n=(1+q^2+q^3)/(1-I_3)$ (for
  $I_3=2q^3+3q^4+2q^5+q^6$ given by Formula (\ref{formulaIn})).
  Since $(m-1-l)\leq\delta(m-1)<\delta g$ (with $g=x_1+\cdots+
  x_{m-1}$ denoting the genus) we get the upper bound
\begin{align}\label{formrightcompdelta}
  &\sum_{n=0}^{\lceil 3\delta g\rceil}r_nq^n
\end{align} 
for contributions
  coming from right-compositions. Coefficients 
of the rational series $\sum_{n=0}^\infty r_nq^n$ have however growth-rate
  $1/\rho>\omega$ for 
  $\rho=0.596\ldots$ the real positive root of $1-I_3$.

  Contributions coming from left compositions are bounded by
  the series $A=\sum_{n=0}^\infty a_nq^n$ defined by
  formula (\ref{formulaleft4}) (which yields upper bounds on left compositions
  with respect to a maximal pivot-part of size $4$).
  Combining this with the contribution (\ref{formrightcompdelta})
we get the upper bound

\begin{align}
 \label{seriesFrobclose4m}
  &\sum_{g=4}^\infty q^g\sum_{i=\lceil g(1-3\delta)\rceil}^ga_ir_{g-i}.
\end{align}
on the generating series for NSG-compositions described by Proposition 
\ref{propmax4delta}. Denoting by $\alpha$ the growth-rate of the 
series $A$ we get the upper bound
$\alpha^{1-3\delta}(1/\rho)^{3\delta}$  on the growth-rate of 
(\ref{seriesFrobclose4m}).  Proposition \ref{propleftcomp4} 
shows the inequality $\alpha<\omega$ which ends the proof since
$\lim_{\delta\rightarrow 0}\alpha^{1-3\delta}(1/\rho)^{3\delta}=\alpha$.
\end{proof}

\subsection{Groups generated by reflections of $\mathbb R$}
\label{sectCayley}

A group $\Gamma$ acting properly by affine
isometries on the $d$-dimensional Euclidean space $\mathbb E^d$
is \emph{crystallographic} if it has a bounded fundamental domain.
Crystallographic groups are considered up to equivalence
under conjugation by affine bijections.
A crystallographic group is a \emph{Bravais group} if it is the full group
of all affine isometries of an Euclidean lattice.

The simplest crystallographic group is $\mathbb Z^d$ (acting by translations).
The group $\mathbb Z^d$ is not a Bravais group.

The simplest non-commutative example is the simplest Bravais group
$\{\pm 1\}\ltimes \mathbb Z^d$
with $\pm 1$ acting by $x\longmapsto \pm x$ on the Euclidean space (and with
$-1$ conjugating a translation to its inverse). It is
the full group of affine isometries of a generic $d$-dimensional
Euclidean lattice having a trivial automorphism group reduced to
$\pm 1$. It consists of all translations by elements of
$\mathbb Z^d$ and of all involutions $x\longmapsto z-x,\
z\in \mathbb Z^d$ with fix-points in the
super-lattice $\frac{1}{2}\mathbb Z^d$ containing $\mathbb Z^d$
with index $2^d$.  

\begin{prop} 
  The group $\langle \mathcal R\rangle$ generated by a finite
  set $\mathcal R$ of real reflections of $\mathbb R$ is isomorphic to the
  simplest Bravais group $\{\pm 1\}\ltimes\mathbb Z^d$
  where $d$ is the dimension of the $\mathbb Q$-vector-space
  generated by all translations (given by $\sigma\circ \rho$ for
  $\sigma,\rho\in \mathcal R$) of $\langle \mathcal R\rangle$.
\end{prop}

\begin{proof} Products of an even number of
  generators in $\mathcal R$ generate the normal commutative
  subgroup $\Gamma^e$ (of index $2$)  
  consisting of all translations in $\Gamma=\langle\mathcal R\rangle$. The torsion-free subgroup
  $\Gamma^e$ is thus isomorphic to $\mathbb Z^d$ where $d$
  is the dimension of the $\mathbb Q$-vector space spanned by
  all translations defined by elements in $\Gamma^e$.
  (The group $\Gamma$ acts of course by affine
  bijections on the underlying $\mathbb Q$-affine space.)
  
  Up to conjugation of $\Gamma$ by the translation of $\mathbb R$
  which sends the fix-point of $\rho_0$ to the origin,
  we can assume that the first
  element $\rho_0$ of $\mathcal R$ is given by $x\longmapsto -x$.
  This element acts on $\Gamma^e$ by conjugating each element of
  $\Gamma^e$ to its inverse.

  Finally, the reflection $x\longmapsto t-x$
  with fix-point $t/2$ for $t\in\mathbb Z^d$ is the composition of
  the reflection $x\longmapsto -x$ followed by the
  translation $x\longmapsto x+t$.
\end{proof}

Given a group $\Gamma=\langle \mathcal G\rangle$ generated by
a finite symmetric set $\mathcal G=\mathcal G^{-1}$ containing all inverses
of its elements, the group $\Gamma$ indexes the set of vertices
of its \emph{Cayley graph} (with respect to the generating set $\mathcal G$)
with edges given by $\{r,gr\}$ for $(g,r)\in\mathcal G\times \Gamma$.
Observe that edges of Cayley graphs are coloured by $g^{\pm 1}$
  for $g$ in $\mathcal G$.

The structure of Cayley graphs is compatible with left-cosets
giving rise to Schreier graphs: The \emph{Schreier graph}
with respect to a subgroup $H$ of $\Gamma$ has 
vertices given by left cosets $rH$ for $r$ in $\Gamma$ and edges given by
$\{rH,grH\}$ for $g$ in $\mathcal G$. This turns sets with actions of $\Gamma$
into graphs: connected components are orbits and correspond to Schreier
graphs with respect to stabilisers of base-points chosen in each orbit.

We apply this to groups $\Gamma=\langle \alpha,\beta,\gamma\rangle$
generated by three reflections $\alpha(x)=a-x,\ \beta(x)=b-x,\ \gamma(x)
=c-x$ where $a,b,c\in \mathbb Z$ are three distinct integers.
$\Gamma$ is always the infinite dihedral group with translation-subgroup
$\mathrm{gcd}(b-a,c-b)\mathbb Z$. Orbits of its obvious action
(by affine isometries) on $\mathbb Z$ are either
isomorphic to its Cayley graph or are isomorphic to Schreier graphs
defined by two-element subgroups fixing some integer\footnote{For graphical
  representations of such orbits the reader can look at pictures of
  single-walled carbon nanotubes.}. 

More precisely, since compositions of all three generators
define reflections, the Cayley graph (which is a bipartite $3$-regular
graph with edges of three ``colours'' corresponding to the sum in
$\{a,b,c\}$ of their two
endpoints) is a hexagonal tiling of
a cylinder $\mathbb S^1\times \mathbb R$. Orbits of
its obvious action on $\mathbb Z$ (with vertices $\mathbb Z$ and edges
$\{x,y\}$ for $x+y\in \{a,b,c\}$) are either hexagonal tilings of cylinders
(for orbits with free action) or of half-cylinders (for orbits with
stabilisers), up to neglecting
a finite set of points near fix-points of generators forming
a somewhat messy 'cap'. The exact nature of the tiling (and of
the extremal cap in the
half-cylinder case) depends on arithmetic properties of $b-a,c-b$ and
the orbit.

The example $\{a,b,c\}=\{1,3,5\}$ leads to a free transitive action on
$\mathbb Z$ (which can hence be identified with the Cayley graph of
$\Gamma$). The tiled cylinder is given by
\begin{align}\label{exple135}
&\begin{array}{ccccccccccccc}
    &\underbar{2}&&0&&2&&4&&6&&8\\
    &\diagup\diagdown&&\diagup\diagdown&&\diagup\diagdown&&\diagup\diagdown&&\diagup\diagdown&&\diagup\diagdown&\\
    7&&5&&3&&1&&\underbar{1}&&\underbar{3}&&\underbar{5}\\
    \vert&&\vert&&\vert&&\vert&&\vert&&\vert&&\vert\\
    \underbar{6}&&\underbar{4}&&\underbar{2}&&0&&2&&4&&6\\
    &\diagdown\diagup&&\diagdown\diagup&&\diagdown\diagup&&\diagdown\diagup&&\diagdown\diagup&&\diagdown\diagup&\\
    &9&&7&&5&&3&&1&&\underbar{1}&\\
 \end{array}
\end{align}
with $\underbar{x}$ representing $-x$. Endpoints of vertical edges
sum up to $1$ (corresponding to the generator $x\longmapsto 1-x$.
The corresponding sum for edges $\diagdown$, respectively $\diagup$,
is $3$, respectively $5$. Edges with
identical endpoints are of course identified in the representation
given above.

\begin{rem}
  Identifying even integers with even translations, the group-law on
  $\Gamma$ is given by $x\cdot y=x+(-1)^xy$
  for $x,y\in\mathbb Z$ (this works if and only if $a,b,c$ are all odd
  and $\mathrm{gcd}(b-a,c-b)=2$, or equivalently, if the action of $\Gamma$
  on $\mathbb Z$ is simply transitive).
  Observe that this group-law is compatible with classes modulo
  $2N$ (leading to finite dihedral groups).
\end{rem}

  An example of an orbit defining a Schreier graph is
  given by $\{a,b,c\}=\{0,1,2\}$:
\begin{align}\label{exple012}
&\begin{array}{ccccccccc}
      &&&&4&&6&&8\\
      &&&&\diagup\diagdown&&\diagup\diagdown&&\diagup\phantom{\diagdown}\\
      &&&\underbar{2}&&\underbar{4}&&\underbar{6}\\
      &&&\vert&&\vert&&\vert\\
      &&&3&&5&&7\\
      &&&\diagup\diagdown&&\diagup\diagdown&&\diagup\diagdown\\
      1&-&\underbar{1}&&\underbar{3}&&\underbar{5}&&\underbar{7}\\
      \vert&&\vert&&\vert&&\vert&&\vert\\
      0&-&2&&4&&6&&8\\
      &&&\diagdown\diagup&&\diagdown\diagup&&\diagdown\diagup\\
      &&&\underbar{2}&&\underbar{4}&&\underbar{6}
      \\
 \end{array}
\end{align}
(with conventions as above).

Each generator $g$ in $\{\alpha,\beta,\gamma\}$ of such a group
$\Gamma=\langle\alpha,\beta,\gamma\rangle$ defines parallel
\emph{lanes} given by bi-infinite sequences $\ldots,H_{-1},H_0,H_1,\ldots$
of hexagons in the Cayley graph of $\Gamma$ with $H_i,H_{i+1}$
sharing a common edge coloured by the generator $g$ (i.e. of
the form $\{r,gr\}$). The three generators of $\Gamma$ correspond thus
to three directions (or types) or parallel lanes.
A lane is of colour $g$ if its interior edges (separating adjacent
hexagons) are coloured by the generator $g$ in $\{\alpha,\beta,\gamma\}$.
A lane is \emph{embedded} if there are no identifications along its
boundaries, i.e. if it
injects into the Cayley graph. A lane is embedded if and only
if it misses some edges (of its proper colour)
in the full Cayley graph of $\Gamma$.
Embedded lanes have parallel lanes of the same direction.
Non-embedded lanes of a given colour (direction) are unique if they
exist.

All lanes are infinite
(otherwise we get a
non-trivial translation of finite order defined by the composition of the two
reflections associated to its boundary-edges).
This implies that there is always a colour corresponding to embedded lanes.
Indeed, all three directions of lanes go off to infinity and lanes with fastest
escape-rates are embedded since they cover a strictly smaller proportion of
the total area than lanes of other colours (which wind more
around the tiled cylinder).

In example (\ref{exple135}), the two parallel
lanes defined by $x\longmapsto 3-x$ (corresponding to edges $\diagdown$) are
embedded. Both remaining generators define a unique lane which is
not embedded.

The notion of a lane makes sense for Schreier graphs defining
hexagonal tilings of half-cylinders by considering infinite sequences
$H_0,H_1,H_2,\ldots$ of distinct hexagons moving only in one direction.
Embeddedness is defined by neglecting the messy end capping off the
tiled half-cylinder. Example (\ref{exple012}) gives rise to
two embedded parallel lanes associated to the generator $x\longmapsto 1-x$
(corresponding to vertical edges)
and to (unique) non-embedded lanes associated to each of the remaining
generators.

\subsection{Spin models on lanes}\label{sectlane}

Recall that a lane is isomorphic to an infinite sequence of consecutively
adjacent hexagons with centres on a straight line. 
The following picture shows 
the $i$-th hexagon $H_i$ of an upgoing vertical lane.
\begin{align}\label{piecelane}
  \begin{array}{ccccccc}
    &&l_{i+1}&-&r_{i+1}\\
    &\diagup&&&&\diagdown\\
    l'_i&&&H_i&&&r'_i\\
    &\diagdown&&&&\diagup\\
    &&l_i&-&r_i\\
    \end{array}
\end{align}
A \emph{finite (embedded) lane} is a connected finite graph defined by a finite
number of consecutively adjacent hexagons in
a (embedded) lane. Finite lanes of vertical direction are obtained by
vertically stacking (\ref{piecelane}) in the obvious way. We use the notations
of (\ref{piecelane}): Left and right boundaries form paths with vertices
$\ldots,l_i,l'_i,l_{i+1}l'_{i+1},\ldots$, respectively
$\ldots,r_i,r'_i,r_{i+1},r'_{i+1},\ldots$.

We consider the constrained spin model with spins in $\{1,2,3\}$
on vertices of finite embedded lanes such that $x_i+x_j\geq 3$ for spins
$x_i,x_j$ on adjacent vertices. We have the following result:

\begin{prop}\label{proppartlane} The partition function of a
  finite embedded lane containing $n$ hexagons is given by the polynomial
\begin{align}\label{defLlane}
L_n&=(q^2+q^3)^{2n}\left(\begin{array}{ccc}1&1&1\end{array}\right)
         T^n\left(\begin{array}{c}
                                             q(q^2+q^3)\\(q^2+q^3)q\\(q^2+q^3)(q^2+q^3)\end{array}\right)
\end{align}
    where
   \begin{align}\label{formulatransferT3}
    T=\left(\begin{array}{ccc}
      q(q+q^2+q^3)&q(q^2+q^3)&q(q+q^2+q^3)\\
      q(q^2+q^3)&q(q+q^2+q^3)&q(q+q^2+q^3)\\
      (q^2+q^3)(q+q^2+q^3)&(q^2+q^3)(q+q^2+q^3)&(q+q^2+q^3)^2
            \end{array}\right)
          \end{align}

   Evaluations at $q=\omega^{-1}$ of the polynomials $L_n$ decay exponentially
          fast to $0$.
\end{prop}

\begin{proof}
  We use transfer-matrices with respect to bases given by
given by
  $$\begin{array}{l||c|c}
    &l_i&r_i\\
    \hline
      b_1&q&q^2+q^3\\
      b_2&q^2+q^3&q\\
      b_3&q^2+q^3&q^2+q^3
    \end{array}$$
    encoding all possibilities for spins at two horizontally
    adjacent vertices
    $l_i,r_i$ belonging to the left and right boundary path
    of a lane, see (\ref{piecelane}).

    The reader should convince himself that $(q^2+q^3)^2T$
    is the transfer matrix with respect to the basis $b_1,b_2,b_3$.
    Denoting by $1,\geq 2$, respectively $\geq 2,1$ and $\geq 2,\geq 2$
    possibilities of spins at $l_i,r_i$ for $b_1,b_2,b_3$,
    the transfer matrix can be recovered from the following table
    with columns indexed by spins at $l_i,r_i$, rows indexed
    by spins at $l_{i+1},r_{i+1}$ and entries giving by all possibilities
    (using the notation $*$ for arbitrary spins in $\{1,2,3\}$ and
    $\geq 2$ for spins in $\{2,3\}$)
    for spins at the intermediary vertices $l'_i,r'_i$:
    $$\begin{array}{l|ccc}
        &1,\geq 2&\geq 2,1&\geq 2,\geq 2\\
\hline
        1,\geq 2&\geq 2,*&\geq 2,\geq 2&\geq 2,*\\
        \geq 2,1&\geq 2,\geq 2&*,\geq 2&*,\geq 2\\
        \geq 2,\geq 2&\geq 2,*&*,\geq 2&*,*\end{array}\ .$$

    Initial conditions are encoded by the final column-vector of 
    (\ref{defLlane}). The initial row vector sums over all possible
    states for the two final vertices.

    Exponentially fast decay to $0$ of $L_n$ at $q=\omega^{-1}$ follows
    from the three eigenvalues (given approximately by
    $0.0916,0.1459$ and $0.9297$) smaller than $1$ of the evaluation
    at $q=\omega^{-1}$ of $(q^2+q^3)^2T$.
\end{proof}

We set 
\begin{align}\label{defMn3}
M_n=\left(\begin{array}{cc}1&1\end{array}\right)
  \left(\begin{array}{cc}0&q\\q^2+q^3&q^2+q^3\end{array}\right)^{2n}\left(
                                       \begin{array}{c}q\\q^2+q^3\end{array}\right)
\end{align}
and
\begin{align}\label{formulaspincyl}
  S_n=\frac{L_n(1+M_n)}{1-L_n}\ .
\end{align}

\begin{prop}\label{propspincyl} Up to contributions of boundary
  vertices, 
  the series $S_n$ defined by formula (\ref{formulaspincyl})
  gives upper bounds for coefficients of the partition function
  (of the constrained spin model with spins in $\{1,2,3\}$) for
  finite hexagonal tilings of bounded cylinders
 tiled by a finite number of parallel embedded lanes consisting of $n$
  hexagons.

  Evaluations at $q=\omega^{-1}$ of $S_n$ decay exponentially fast to $0$.
\end{prop}

\begin{proof} We consider a finite graph $\Gamma$ on a hexagonally tiled
  cylinder
which consists of $k$ embedded parallel finite lanes all consisting of $n$
hexagons.
Since the family of $k$ lanes consists of embedded lanes we have $k\geq 2$.
Up to a few boundary vertices, vertices of $\Gamma$ are contained
in the union of $k/2$ embedded disjoint lanes if $k$ is even and
they are contained in the union
of $(k-1)/2\geq 1$
disjoint embedded lanes and of a simple path (parallel to the boundaries of
the previous lanes) of length $2n$ if $k$ is odd.
The fraction $\frac{L_n}{1-L_n}$
gives upper bounds for the contribution coming from the $\lfloor k/2\rfloor$
disjoint embedded lanes (consisting of $n$ hexagons).
The factor $(1+M_n)$ accounts for the possible presence (for $k$ odd)
of a simple path of length $2n$.

The evaluation at $q=\omega^{-1}$ of the transfer matrix involved in
the definition of $M_n$ given by formula (\ref{defMn3}) has eigenvalues
$1$ and $-\omega^{-2}$. Evaluations of $M_n$ at
$q=\omega^{-1}$ define thus a bounded sequence.

Since $L_n$ decays exponentially fast to $0$, the final part of Proposition
\ref{propspincyl} holds.
\end{proof}

\begin{rem}
  The transfer-matrix
  $\left(\begin{array}{cc}0&q\\q^2+q^3&q^2+q^3\end{array}\right)$
involved in the definition of $M_n$ has 
  eigenvalue $1$ at $q=\omega^{-1}$. This is of course the reason for
  working with more complicated graphs (given by
  hexagonal tilings and associated to three maximal parts)
  when dealing with NSG-compositions of maximum $4$. 
\end{rem}


\begin{rem} The spin models considered in this section
  can be considered as a variation of the hard-square model
  (with spins in $\{0,1\}$ on vertices of the square lattice such
  that vertices with maximal spin $1$ are isolated). We work essentially
  with spins
  in $\{1,2,3\}$ on vertices of quotients of the $3$-regular
  graph defined by a hexagonal tiling such that vertices with
  minimal spin
  $1$ are isolated.
\end{rem}

\subsection{The  $\epsilon$-condition}\label{subsectepsilon4}

Given $\epsilon>0$, a NSG-composition
$x_1+\cdots+x_{m-1}$ of maximum $4$
satisfies the $\epsilon$-condition if it has at least three maximal
parts and if we have $l_3-l_1\leq\epsilon(m-1-l_1)$ for
the indices $l_1<l_2<l_3$ of the last three maximal parts.

\begin{prop}\label{propepsiloncond} There exists $\epsilon>0$ such that
  NSG-compositions of maximum $4$ satisfying the $\epsilon$-condition
  have growth-rate strictly smaller than $\omega$.
\end{prop}

\begin{proof}
    We consider a NSG-composition $x_1+\cdots+x_{l_1}+\cdots  +x_{m-1}$ satisfying the $\epsilon$-condition
  (for some small enough positive $\epsilon$) with $x_{l_1}=x_{l_2}=x_{l_3}=4$
  for $l_1<l_2<l_3$ the last three maximal parts, as in the definition
  of the $\epsilon$-condition.
  We use pivot-factorisation with respect to the third-last maximal
  part $x_{l_1}$.
  
Proposition \ref{propleftcomp4} shows that 
  associated left compositions have growth-rate strictly
  smaller than $\omega$.

  In order to study right compositions, 
  we consider the graph $\Gamma$ with $m-1-l_1$
  vertices labelled $l_1+1,\ldots,m-1$
  and edges $\{i,j\}$ if $i+j-m\in\{l_1,l_2,l_3\}$.
  More precisely, the graph $\Gamma$ is the restriction to
  $\{l_1+1,\ldots,m-1\}$ of the orbit-graph underlying the
  group generated by the three real reflections $x\longmapsto m+l_i-x$
  with half-integral fix-points $(m+l_i)/2$ in $\frac{1}{2}\mathbb N$.
  Removing from $\Gamma$ the central set $\{\lfloor (l_1+m)/2\rfloor,\ldots,
  \lceil l_3+m)/2\rceil\}$ (called the set of \emph{central vertices})
  containing all fix-points, the graph $\Gamma$
  is a finite union of hexagonal tilings on cylinders
  (up to neglecting a few incomplete hexagons at boundaries),
  see Section \ref{sectCayley}. These hexagonal
  tilings contain parallel lanes of proper type
  consisting of a roughly equal number of at least $(m-l_1)/(2(l_3-l_1))-c$
  hexagons for some small
  constant $c$ (choosing $c=10$ certainly works). For $\epsilon$ small enough,
  we have thus a lower bound of $\frac{1}{4\epsilon}$ for the number of
  hexagons of all lanes.
  
  Proposition \ref{propspincyl} shows thus that the contribution of
  most vertices of $\Gamma$ (except central vertices and a few vertices
  near the other boundaries of these finite cylinder-tilings)
  can be bounded by $\sum_{n=\lfloor 1/(4\epsilon)\rfloor}^\infty
  \frac{S_n}{1-S_n}$
(for $S_n$ given by (\ref{formulaspincyl}))
  which has growth-rate strictly smaller than $\omega$ by Proposition
  \ref{propspincyl} and Lemma \ref{lemconvradius}.

  The contribution of the remaining $a=O(l_3-l_1)$ vertices can be bounded
  trivially by $(q+q^2+q^3)^aq^2$ (with the factor $q^2$ accounting
  for the values $x_{l_2}=x_{l_3}=4\not\in\{1,2,3\}$). Such vertices
  contribute thus at most a proportion of $\epsilon$ to the (degree of)
  the total contribution of such a NSG-composition.
  We can thus apply the arguments
  used for the series given by (\ref{seriesFrobclose4m}) during the proof of
  Proposition  \ref{propmax4delta}.
\end{proof}

\subsection{Proof of Proposition \ref{propmax4}}\label{sectproofpm4}

The main ingredient for proving Proposition \ref{propmax4}
is the following result:

\begin{prop}\label{propmax4often} There exists a natural integer $A$ such that
  NSG-compositions with at least $A$ parts of maximal size $4$
  have growth-rate strictly smaller than $\omega$.
\end{prop}

\begin{proof}\label{propAparts} We $\delta$
  such that Proposition \ref{propmax4delta} holds.
  We chose $\epsilon$ such that Proposition \ref{propepsiloncond}
  holds for $\epsilon'=\frac{\epsilon}{1-\epsilon}$.
  We choose now a natural integer $N$ such that $(1-\epsilon)^N<\delta$.
  We are going to prove that $A=2N+1$ works.

  Let $\mathbf x=x_1+\cdots+x_{m-1}$ be a NSG-composition having at least $A=2N+1$
  maximal parts of size $4$.
  For $i=0,\ldots,N-1$ we denote by $I_i$ the set of all integers
  of the real interval $$[\lceil m-1-(1-\epsilon)^i(m-1)\rceil,
  \lfloor m-1-(1-\epsilon)^{i+1}(m-1)\rfloor]\ .$$ We set $I_N=
  [\lceil m-1-(1-\epsilon)^N(m-1)\rceil,m-1]\cap \mathbb N$.
  We have of course $\cup_{i=0}^N I_i=\{1,\ldots,m-1\}$.
We denote by $\mathcal L$ the set of the $A=2N+1$ largest elements in the set
$\{i\ \vert 1\leq i\leq m-1,x_i=4\}$ indexing maximal parts of $\mathbf x$.
Proposition \ref{propmax4delta} and the definition of $I_N$
show that the set of such NSG-compositions with $\mathcal L$ intersecting $I_N$
has growth-rate smaller than $\omega$.

We can thus assume that $\mathcal L$ does not intersect $I_N$.
The pigeon-hole principle implies thus that there exists a set $I_j$ 
among the $N$ sets $I_0,\ldots,I_{N-1}$ containing at least three 
elements of $\mathcal L$. Let $l_1<l_2<l_3$ be the three largest elements
of $\mathcal L\cap I_j$. We denote by $a=\sharp\{i\in\mathcal L\ \vert\
i>l_3\}\leq A-3$ the number of elements of
$\mathcal L$ which are larger than $l_3$. 
  Adjoining $a$ times the Frobenius element to (the numerical semigroup
  associated to) $\mathbf x$, we get
  a NSG-composition $\underline{\mathbf x}$
  of genus $g-a$ with parts $\underline{x}_i=x_i$ if $i\leq l_3$
  and $\underline{x}_i=\min(3,x_i)$ if $i>l_3$. 
  Since three last maximal parts
  $\underline{x}_{l_1}=\underline{x}_{l_2}=\underline{x}_{l_3}=4$ of
  $\underline{\mathbf x}$ have indices $l_1,l_2,l_3$ in $I_j$ we get
  $$l_3-l_1\leq \frac{(1-(1-\epsilon))}{1-\epsilon}(1-\epsilon)^{i+1}(m-1)\leq \frac{\epsilon}{1-\epsilon}(m-1-l_1)$$
  which implies that $\underline{\mathbf x}$ satisfies the
$\epsilon'$-condition for $\epsilon'=\frac{\epsilon}{1-\epsilon}$.
  The generating series of such NSG-compositions is bounded by
  the series
  $$\sum_{k=0}^{A-3}q^{2k}\frac{d^k}{dq^k}G_{\epsilon'}(q)$$
  with $G_{\epsilon'}$ denoting the generating series for all NSG-compositions of
  maximum $4$ satisfying the $\epsilon'$-condition.
  Since $G_{\epsilon'}$ has growth-rate strictly smaller than
  $\omega$ for $\epsilon'$ small enough
  (see Proposition \ref{propepsiloncond}) and since
  algebraic differential operators do not increase the growth-rate we get
  the result.
\end{proof}

\begin{proof}[Proof of Proposition \ref{propmax4}]
Combine Propositions \ref{propmax4often} and \ref{propmax4Aparts}.
\end{proof}

\section{NSG-compositions with maximum $3$}\label{sectmax3}

\begin{prop}\label{propmax3} The generating series for the
number of all NSG-compositions with maximum $3$ is given by 
$$\frac{\sum_{j=3}^\infty \tilde c_jq^j}{1-q-q^2}$$
where $\tilde C=\sum_{g=3}^\infty\tilde c_gq^g$ is the generating series
of all NSG-compositions with maximum $3$ ending with a maximal part.
\end{prop}

\begin{rem} The number $\tilde c_g$ is the number of numerical semigroups
$S$ of genus $g=\sharp(\mathbb N\setminus S)$ with Frobenius number 
$f=3m-1$.
\end{rem}

\begin{proof}[Proof of Proposition \ref{propmax3}]
All NSG-inequalities 
$x_{m-s-t}\leq x_{m-s}+x_{m-t}+1$ (with $1\leq s,t\leq s+t<m$)
of the second line in
(\ref{fundeqsg}) are satisfied for compositions with maximum at most $3$.

Any NSG-composition $x_1+\ldots+x_{m-1}$ of maximum $3$
has thus a unique pivot-factorisation
$$(x_1+\cdots+x_{l-1}+3)+(x_{l+1}+\cdots+x_{m-1})$$
with pivot-part the last occurrence $x_l=3$ of a maximal part
into a left composition defining a NSG-composition with parts in $\{1,2,3\}$
ending with the pivot part $x_l=3$ (defined as the last maximal part)
and a right composition defining a NSG-composition with
all parts in $\{1,2\}$. This decomposition is bijective:
Concatenating a NSG-composition ending with a maximal part $3$
with a composition having all parts in $\{1,2\}$ yields an 
NSG-composition of maximum $3$.

The result follows since the numerator accounts for the initial
NSG-composition ending with $3$.
Possibilities for the final composition with all parts 
in $\{1,2\}$ are enumerated by the denominator, see
Proposition \ref{propparts12}.
\end{proof}

\subsection{Convergencency of $\tilde C$ at $\omega^{-1}$}
\label{secttildeC}

\begin{thm}\label{thmradiustildeC} 
  The generating series $\tilde C=\sum_{j=3}^\infty\tilde c_jq^j$
  for NSG-compositions ending with a maximal part of size $3$
  has growth-rate strictly smaller than $\omega$.
\end{thm}

  \begin{proof}
    Given $\epsilon>0$,
  a NSG-composition $x_1+\cdots +x_{m-1}$
  ending with a maximal part $x_{m-1}=3$ satisfies
  the $\epsilon$-condition if it has at least three maximal parts and
  if we have $(m-k)<\epsilon m$ for the third last maximal part
  $x_k=3$.

  A slight modification of the proof of Proposition
  \ref{propepsiloncond} shows that NSG-compositions
  satisfying the $\epsilon$-condition have growth-rate
  smaller than $\omega$ for $\epsilon$ small enough.
  We fix $\epsilon$ to a strictly positive value which
  is small enough.
    
  We consider now NSG-compositions ending with a maximal part $x_{m-1}=3$
  which do not satisfy the $\epsilon$-condition (for our fixed value
  of $\epsilon$).
  Adding to (the numerical semigroup of)
  such a composition the second largest gap if it
  corresponds to an index $l$ at least equal to $(1-\epsilon)m$,
  we are reduced to consider NSG-compositions $x_1+\cdots+x_{m-1}$
  of maximum $3$ ending with a maximal part such that we have
  $k<(1-\epsilon)m$ for the (perhaps non-existing)
  index $k$ of the second last maximal part.
  
  We consider the graph $\Gamma$ with vertices $1,\ldots,m-2$ and edges
  $\{i,j\}$ if the sum $i+j$ belongs to $
  \{k,m-1\}$ (with $k$ missing if $x_{m-1}$ is the
  unique maximal part of $x_1+\cdots+x_{m-1}$. Connected components are
  paths of lengths at most equal to $2/\epsilon+1$ with endpoints in
$k,k+1,\ldots, m-2$, except for (perhaps existing)
  exceptional components ending at an element of
  $\{k/2,(m-1)/2\}\cap \mathbb N$.
The associated
  constrained Ising model has spins in $\{1,2,3\}$
except for endpoints $k+1,\ldots,
  m-2$ corresponding to spins in $\{1,2\}$.
  (If such a NSG-composition has a unique part of maximal size $3$,
  these paths are simply of length $1$, given by edges $\{m-i,i\}$,
  together with the (not necessarily existing) isolated vertex $(m-1)/2$.)
  The generating function for such paths of length $n$ is given by
  $$\tilde P_n=\left(\begin{array}{cc}1&1\end{array}\right)
  \left(\begin{array}{cc}0&q\\q^2+q^3&q^2+q^3\end{array}
  \right)^n\left(\begin{array}{c}q\\q^2\end{array}\right)$$
  and we have $\tilde P_n(\omega^{-1})\leq \tilde P_2(\omega^{-1})=
  \frac{13-5\sqrt{5}}{2}<1$ for all $n\geq 1$.
  Observe now that the number of NSG-compositions under consideration
  is bounded by the coefficients of the rational function 
  $$\left(1+q^2\frac{d}{dq}\right)\left((1+q^2+q^3)(1+q)
    \left(\sum_{j=1}^{\lceil 2/\epsilon+1\rceil}\frac{1}{1-\tilde P_j}\right)^4
  q^3\right)$$
with convergency radius strictly larger than $\omega^{-1}$.
(The differential operator accounts for a perhaps suppressed second-last
maximal part, the factor $(1+q^2+q^3)$ takes into account a perhaps
existing isolated vertex $(m-1)/2$ of $\Gamma$, the factor $(1+q)$
corrects for a (perhaps non-existing)
path with initial vertex starting at the second-last maximal part with
index at least $(1-\epsilon)m$,
the fourth power accounts for all possible paths of length
at least $1$ (there are at most four different possible lengths,
all bounded by $2/\epsilon+1$), the final factor $q^3$ corresponds of
course to the part $x_{m-1}=3$.)
  \end{proof}

  \begin{rem} The upper bound $2/\epsilon+1$ on lengths of paths in $\Gamma$
    is crucial in the previous proof: The evaluations $\tilde P_n(\omega^{-1})$
    tend to a strictly positive limit for $n\rightarrow\infty$
    (the corresponding transfer-matrix has an eigenvalue $1$ at $q=\omega^{-1}$).
    \end{rem}

\section{Proof of Theorem \ref{thmupperboundng}}\label{sectmainproof}
\begin{prop}\label{propgsatmost3}
The generating series for all NSG-compositions
with maximum at most $3$ is given by 
$$\frac{1+\tilde C}{1-q-q^2}$$
where $\tilde C$ is the generating series for all NSG-compositions
ending with a maximal part of size $3$.
\end{prop}

\begin{proof} Follows from Proposition \ref{propparts12} and
  Proposition \ref{propmax3}.
\end{proof}

\begin{cor}\label{corn3g} We have 
  $$\lim_{g\rightarrow\infty}\frac{n_{\leq 3}
    (g)}{\left(\frac{1+\sqrt{5}}{2}\right)^g}
=\frac{5+\sqrt{5}}{2}\left(1+\tilde C((\sqrt{5}-1)/2)\right)$$
for the number $n_{\leq 3}(g)$ of NSG-compositions 
with genus $g$ and maximum at most $3$.
\end{cor}

\begin{proof}[Proof of Corollary \ref{corn3g}]
The algebraic identity
$$\frac{1}{1-q-q^2}=\frac{5+\sqrt{5}}{10}\frac{1}{1-\frac{1+\sqrt{5}}{2}q}+
\frac{5-\sqrt{5}}{10}\frac{1}{1-\frac{1-\sqrt{5}}{2}q}$$
and Theorem \ref{thmradiustildeC} show that 
$$\frac{1+\tilde C(q)}{1-q-q^2}-\frac{5+\sqrt{5}}{10}\ 
\frac{\left(1+\tilde C(\omega^{-1})\right)}{1-\omega q}$$
is holomorphic in a open disc of radius strictly larger than $\omega^{-1}$.
We have thus
$$n_{\leq 3}(g)=\frac{5+\sqrt{5}}{10}\left(1+\tilde C(\omega^{-1})\right)
\omega^g(1+o(1/g))
$$
for coefficients $n_{\geq 3}(g)$ of $(1+\tilde C(q))/(1-q-q^2)$ enumerating NSG-partitions
of genus $g$ with parts in $\{1,2,3\}$.
\end{proof}

\begin{proof}[Proof of Theorem \ref{thmupperboundng}]
  Corollary \ref{corn3g} implies that it is enough to show that
  NSG-compositions of maximum $\geq 4$ have growth-rate strictly smaller
  than $\omega$.
Theorem \ref{thmradiustildeC} and Proposition \ref{propmax4} 
ensure a growth-rate strictly smaller than $\omega$ for
NSG-compositions with maximum $4$. By Proposition \ref{propmax5}
we get then the result for NSG-compositions with maximum $5$
and Proposition \ref{propmaxgeq6} completes the proof.
\end{proof}

\section{NSG-compositions ending with a maximal
  part of size $3$}\label{sectNSGcomp3}

A good understanding of NSG-compositions ending with a maximal part
of size $3$ associated to the generating series $\tilde C$ is desirable
in view of Theorem \ref{thmupperboundng} and Proposition
\ref{propmax3}. This Section has four distinct parts:

In a first part we prove the following result:

\begin{thm}\label{thmlowerboundtildeC}
  The coefficient $\tilde c_g$ of the generating series $\tilde C$
  is at most equal to the coefficient $\alpha_g$
  of the rational series $A=\sum_{g=3}^\infty \alpha_g q^g$ defined by
  \begin{align}\label{seriesalphag}
    A&=\frac{1+q^2+q^3}{1-(q^2+q^3)(q+q^2+q^3)}q^3
                                    \ .
  \end{align}
  In particular, the convergency radius of $\tilde C$ is at most
  equal to the convergency radius $\rho_A=0.659982\ldots$
  of $A$ given by the positive real root $\rho$ of
  $1-(q^2+q^3)(q+q^2+q^3)$.
\end{thm}

In a second part we list a few initial coefficients of
$\tilde C$.

In a third part we give generating sequences for all
contributions to $\tilde C$ with only one or two maximal parts
(of size $3$).

The fourth part is speculative and non-rigourous: We believe that
the upper bound $\rho_A$ given by Theorem \ref{thmlowerboundtildeC}
on the convergency radius on $\tilde C$
is sharp and we describe 
(in a non-rigourous way) a family of
NSG-contributions ending with a maximal part $3$ which should
give the bulk of contributions to $\tilde c_g$ for $g\rightarrow\infty$.

\subsection{Proof of Theorem \ref{thmlowerboundtildeC}}

\begin{proof} We consider the set of all compositions
  $x_1+\ldots+x_{m-1}$ ending with a maximal part
  $x_{m-1}=3$, with $x_1,x_2,\ldots,x_{\lfloor (m-1)/2\rfloor}$
  in $\{2,3\}$ and with $x_{\lfloor (m-1)/2\rfloor+1},
  \ldots,x_{m-2}$ in $\{1,2,3\}$. If $i$ and $j$ are two strictly
  positive integers such that $i+j\leq m-1$, then $\min(i,j)\leq
  \lfloor (m-1)/2\rfloor$ and we have thus $x_i+x_j\geq 2+1=3$
  showing that all such compositions are NSG-compositions ending with
  a maximal part $3$. Regrouping parts $x_i,x_{m-1-i}$ and taking
  into account the central part $x_{(m-1)/2}$ (which exists only
  for odd $m$) we get the generating series
  $$A=\frac{1+(q^2+q^3)}{1-(q^2+q^3)(q+q^2+q^3)}q^3$$
  for the set of all such NSG-compositions.
  This proves the inequality $\alpha_g\leq \tilde c_g$.

  The convergency radius of $A$ is given by the smallest pole
  at the positive real root $\rho_A$ of the denominator of the
  rational series $A$. Since all real positive
  coefficients of $\tilde C$ are bounded below by the corresponding
  coefficients of $A$, the convergency radius of $\tilde C$ is
  at most equal to the convergency radius $\rho_A$ of $A$.
\end{proof}

\subsection{A few initial coefficients for $\tilde C$}

The following tables list a few numbers $\tilde c_g$ counting NSG-compositions
of genus $g$ ending with a maximal part of size $3$. 
The first table contains also the corresponding compositions,
encoded as words in $\{1,2,3\}^*$ (with omitted addition-signs):
$$\begin{array}{l|r|l}
g&\tilde c_g&\hbox{compositions}\\
\hline
3&1&3\\
4&0&\\
5&1&23\\
6&3&33,123,213\\
7&2&223,313\\
8&4&233,323,1223,2213\\
9&9&333,1233,2223,2313,3213,11223,12123,21213,22113\\
10&12&
2233,2323,3223,3313,12223,\\
&&21223,21313,22123,22213,23113,31213,32113
\end{array}$$

Coefficients $\tilde c_{1},\ldots,\tilde c_{50}$
of $\tilde C$ are given by
$$\begin{array}{r|rrrrrrrrrr}
  1-10&   0&0&1&0&1&3&2&4&9&12\\
 11-20&20&32&50&84&132&208&331&526&841&1333\\
\end{array}$$
$$\begin{array}{l|rrrrr}
  21-25& 2145& 3401& 5314& 8396&13279\\
  26-30& 20952& 33029& 51927& 81527& 128102\\
  31-35& 201700& 317461& 498911& 782868& 1226255\\
  36-40& 1919070& 3000905& 4687213& 7315975&11419861\\
  41-45& 17833383& 27857264& 43511423& 67908811& 105857661\\
  46-50&164837336& 256493732& 398937594& 620308837& 964299016
  \end{array}$$
  A sequence starting with initial coefficients of $\tilde C$ does
  presently not appear in \cite{OEIS}.
  
\begin{figure}[h]\label{figdihex}
\epsfysize=8cm
\center{\epsfbox{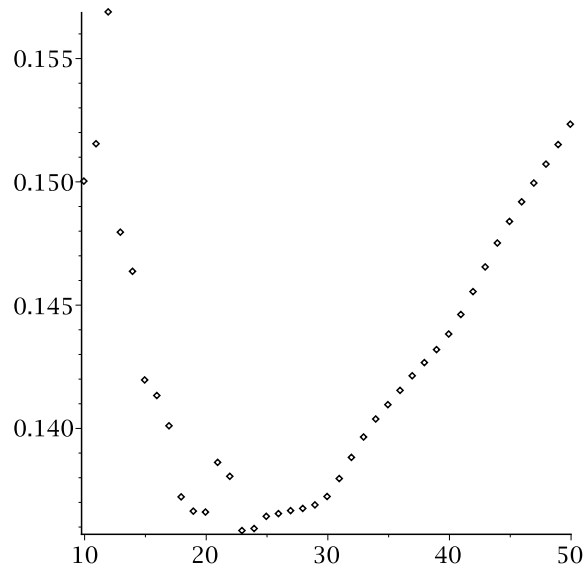}}
\caption{Quotients
  $\tilde c_i/(i a_i)$ for $i=10,\ldots,50$.}
\end{figure}  
  These coefficients were computed using the algorithm outlined
  in Section \ref{sectalgotildec}. Figure 2 displays some
  quotients $\tilde c_i/(i a_i)$ for $\sum_{n=3}^\infty a_nq^n$
  the rational series defined by Formula (\ref{seriesalphag})
of Theorem \ref{thmlowerboundtildeC}.

  Initial coefficients for $\tilde C$
  suggest the inequalities
  $$3.57<C<3.93$$
  for the constant $C$ occuring in Formula (\ref{formulaasymptsemigrps})
  of Theorem \ref{thmupperboundng}. The left inequality is
  proven since it was obtained by considering a large subset
  of NSG-compositions contributing to $\tilde C$.
  The right side is somewhat conjectural: It was obtained by
  assuming $\tilde c_i/\tilde c_{i-1}<\tilde c_{50}/\tilde c_{49}$
  and by replacing the missing coefficients $\tilde c_i$ for $i>50$
  with $\tilde c_{50}\left(\frac{\tilde c_{50}}{\tilde c_{49}}\right)^{i-50}$.

  \subsection{Contributions to $\tilde C$ with $k$ maximal parts}

  We denote by $\tilde c_{g,k}$ the number of NSG-compositions of genus
  ending with a maximal part of size $3$ and having a total number
  of $k$ maximal parts.

  Setting
\begin{align}\label{formtauk}
  \tau(k)&=\sum_{g}\tilde c_{g,k}\omega^{-g}
\end{align}
  we have $\tilde C(\omega^{-1})=\sum_{k=1}^\infty\tau(k)$ and
  Theorem \ref{thmupperboundng} implies immediately the following result:
  
  \begin{prop} The proportion of NSG-compositions with $k$ maximal
    parts of size $3$ is asymptotically given by
    $$\frac{\tau(k)}{1+\sum_{k=1}^\infty \tau(k)}\ .$$
  \end{prop}

  We describe in this section the easy value of $\tau(1)$ and we
  give a useful formula for the already complicated case of $\tau(2)$.

  \begin{prop}\label{propCtildeuniqemaxpart} We have
    \begin{align}\label{seriesuniquemaxpart3}
      \sum_{g=3}^\infty \tilde c_{g,1}q^g&=\frac{1+q^2}{1-(2q^3+q^4)}q^3\ .
    \end{align}
    We get in particular the evaluation $$\tau(1)=\frac{1}{\omega}+
    \frac{1}{\omega^3}=0.85410196624968454461376\ldots\ .$$
  \end{prop}
  
  \begin{proof} We consider graphs with vertices
    $1,\ldots,m-2$ and edges ${i,m-1-i}$. Connected components
    are isolated edges and perhaps an isolated vertex $(m-1)/2$.
    We consider the spin model with spins $1$ or $2$ summing up at
    least to $3$ along edges. The isolated
    vertex $(m-1)/2$ (which exists only for odd $m$) is required
    to have spin $2$. The corresponding partition function,
    corrected by the final factor $q^3$ in order to account for
    the last maximal part $x_{m-1}=3$ is easily seen
    to be given by (\ref{seriesuniquemaxpart3}).
  \end{proof}

  \begin{prop}\label{propserkappa2} The generating series $\sum_{g=6}^\infty \tilde c_{g,2}q^g$
    counting all NSG-compositions ending with a last maximal part of size
    $3$ and having a unique additional maximal part is given by
    \begin{align*}
      &\frac{1+q^3+q^4}{1-(3q^6+4q^7+q^8)}
        q^3\left(1+(q+q^2)\frac{1+q^2}{1-(2q^3+q^4)}\right)q^3\\
      &+q^4\sum_{n=1}^\infty P_{n,0,0}\\
      &+q^4\sum_{n=2}^\infty\frac{P_{2n-1,0,1}(1+P_{n-1,1,0})}{1-P_{2n-1,1,1}}\\
      &+q^4\sum_{n=3}^\infty\left(\frac{P_{2n-1,1,1}+P_{n-1,1,0}}{1-P_{2n-1,1,1}}\right)
        P_{n-2,0,0}\\
      &+q^4\sum_{n=3}^\infty\left(\frac{P_{2n-1,1,1}+P_{n-1,1,0}}{1-P_{2n-1,1,1}}\right)
        \frac{P_{2n-3,0,1}(1+P_{n-2,1,0})}{1-P_{2n-3,1,1}}
    \end{align*}
    where
   $$P_{n,\epsilon_\alpha,\epsilon_\omega}=
  \left(\begin{array}{cc}\epsilon_\omega&1\end{array}\right)
  \left(\begin{array}{cc}0&q\\q^2&q^2\end{array}\right)^n
  \left(\begin{array}{c}\epsilon_\alpha q\\q^2\end{array}\right)\ .$$
  We have in particular
  $$\tau(2)=.7628736853796206184361443135239953344793590\ldots\ .$$
\end{prop}
Coefficients $\tilde c_{g,2}$ for $g=1,\ldots,20$ are given by
$$0, 0, 0, 0, 0, 1, 1, 2, 3, 7, 10, 11, 25, 38, 43, 75, 123, 153, 233, 383\ .$$

\begin{proof}[Sketch of proof for Proposition \ref{propserkappa2}]
  Let $x_1+\cdots+x_k+\cdots+x_{m-1}$ be a NSG-composition
  with $x_k=x_{m-1}=3$ and all remaining parts in $\{1,2\}$.
  We assume first that $2k\leq m-1$. Removing all 'central' parts
  $x_{k+1},\ldots,x_{m-2-k}$ following $x_k$ reduces such a
  NSG-composition to a NSG-composition with $m-1=2k$. The associated
  graph with vertices $1,\ldots,m-2$ and edges $\{i,j\}$ such that
  $i+j\in\{k,2k\}$ gives rise to a
  spin-model with partition function $(1+q^3+q^4)q^6/(1-3q^6+4q^7+q^8)$.
  Possibly removed central parts account for a factor of
  $(1+(q+q^2)(1+q^2)/(1-(2q^3+q^4))$.

  We can now suppose $2k>m-1$. We consider the associated graph $\Gamma$
  with vertices
  $1,\ldots,m-2$ and edges $\{i,j\}$ with $i+j\in\{k,m-1\}$.
  Connected components of this graph are line-segments of length $2n-1,2n-3$
  and perhaps an exceptional line-segment of length $n-1$ and an
  exceptional line-segment of length $n-2$.
More precisely, connected components of length $2n-1$ start at the
last vertices of $\Gamma$ and are 'wrapped around' a (not necessarily
existing) exceptional component of length $n-1$. Components of length
$2n-3$ wrapped around the (not necessarily existing) exeptional
component of length
$n-2$ have the same structure.

  Exceptional line-segments end at a vertex of spin $2$. The connected
  component with endpoint $k$ corresponds also to a spin $2$ point
  after diminishing $x_k=3$ by $1$. Endpoints with spins restricted to
  $2$ are accounted for by setting one or both of the parameters
  $\epsilon_\alpha,\epsilon_\beta$ in the polynomials $P_{n,\epsilon_\alpha,\epsilon_\omega}$ (counting possibilities for spins in $\{1,2\}$ summing up at least to $3$ along edges in a line-graph of length $n$) to $0$.
  The different combinatorial
  possibilities correspond to the next four summands.
 The somewhat lengthy but straightforward details are left to the reader.
 \end{proof}

\subsection{Speculations on typical asymptotic
  contributions to $\tilde C$}

We say that a class of NSG-compositions ending with a last maximal
part of size $3$ is a \emph{asymptotically $\tilde C$-typical}
if the proportion of NSG-compositions of the class contributing to
$\tilde c_g$ tends to $1$ for $g\rightarrow\infty$.

The aim of this Section is to describe a class which should
be asymptotically $\tilde C$-typical. The class is
given by the set of all NSG-compositions close to
all NSG-compositions occuring in the proof of Theorem
\ref{thmlowerboundtildeC}. Somewhat informally,
an asymptotically $\tilde C$-typical NSG-composition
$x_1+\cdots+x_{m-1}$ has two large regular chunks:
The first chunk starts with $x_1$ and ends somewhere slightly
before $x_{\lfloor(m-1)/2\rfloor}$. Its parts are independent
random variables, equal to $2$
with probability $1/(1+\rho_A)$ (for $\rho_A$ as in Theorem
\ref{thmlowerboundtildeC})
and equal to $3$ with probability $\rho_A/(1+\rho_A)$.
It contains no parts of size $1$.
This part is followed by a small transitional region
centered around $(m-1)/2$ (containing parts $1$ with gradually
increasing density) ending at the beginning of the
second large chunk with parts given by independent
random variables equal to $\alpha\in\{1,2,3\}$
with probabilities $\rho_A^{\alpha-1}/(1+\rho_A+\rho_A^2)$.
The very last parts of an asymptotically $\tilde C$-typical
NSG-composition form again a transitional region containing parts of size $3$
with gradually decreasing density. (The presence of pairs $x_i=x_j=1$ of parts
$1$ with indices $i+j<m-1$ in the central
transitional region forces $x_{i+j}\leq 2$.) 

An asymptotically $\tilde C$-typical NSG-composition $x_1+\ldots+x_{m-1}$
of large multiplicity $m$ 
has genus $\frac{1}{2}\left(\frac{2+3\rho_A}{1+\rho_A}+
  \frac{1+2\rho_A+3\rho_A^2}{1+\rho_A+\rho_A^2}\right)m+O(\sqrt{m})$.

We ignore the typical size of the central and final transitional parts.

\section{Generic numerical semigroups}\label{sectgennsg}

We call NSG-compositions with maximum at most $3$ \emph{asymptotically
  generic} or simply \emph{generic} since they prevail proportionally
in large genus.

\subsection{A combinatorial model}\label{sectcombtree}

Generic NSG-compositions (NSG-compositions of maximum at most $3$)
can be extended by adding a last part $1$ or $2$.
Depending on the configuration of parts equal to $1$ an extension by a last
part of size $3$ is sometimes possible.
We encode this by a rooted binary tree with left descendants
corresponding to parts $1$, right descendants corresponding to
parts $2$ or (sometimes) $3$. We represent this by drawing thin edges for
right descendants corresponding only to extensions by a part of size
$2$ and by drawing fat
edges for right descendants corresponding to extensions by an
additional last part of size  $2$ or $3$.
A finite downward path starting at the root represents $2^f$
generic NSG-compositions if it contains $f$ fat edges.
All contributions to $\tilde C$ can be computed as follows.
A given fat edge (representing a final part of size $3$)
joined by $f$ (different) fat edges,
$l$ left edges and $s$ slim right edges to the root corresponds
to $2^f$ generic NSG-compositions ending with a last part $3$
yielding a total contribution of $q^{3+l+2s+2f}(1+q)^f$ to $\tilde C$.

\begin{figure}[h]\label{figtree}
\epsfysize=4.5cm
\center{\epsfbox{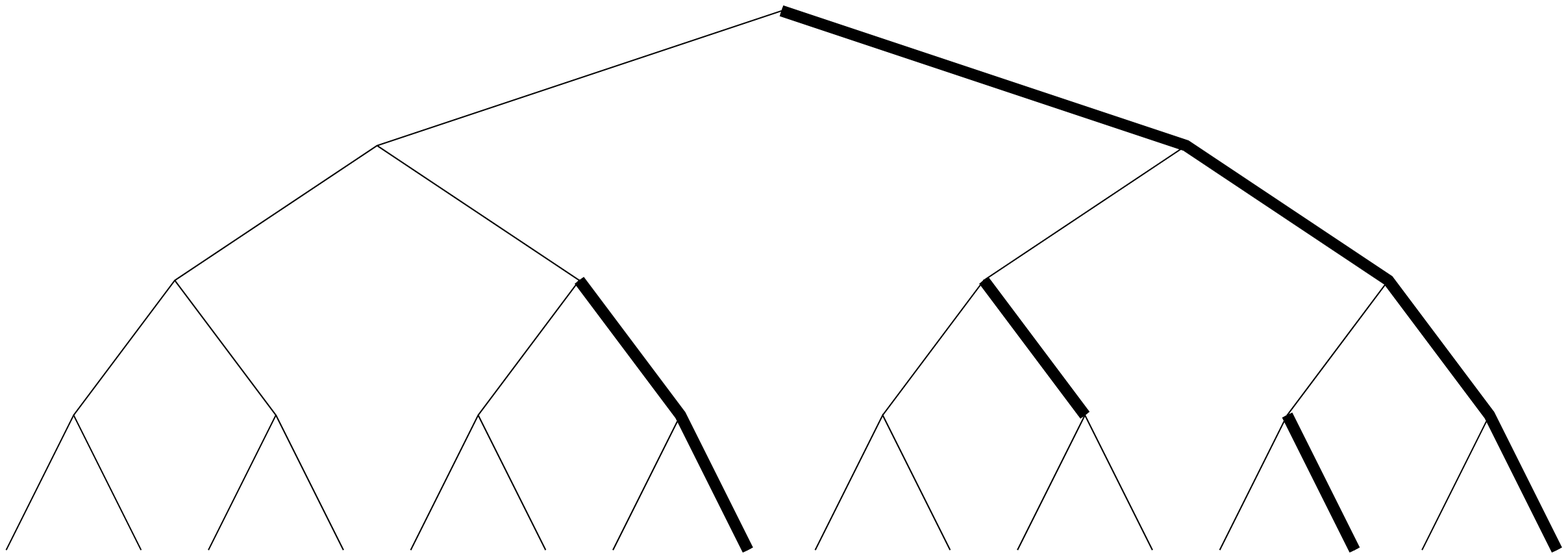}}
\caption{The tree of generic NSG-compositions.}
\end{figure}
\subsection{An algorithm of complexity $3^{2g/3}$ for $\tilde c_g$}
\label{sectalgotildec}

The tree of generic NSG-compositions suggests the following elementary
algorithm of complexity $3^{2g/3}$ for computing $\tilde c_g$.

Let $x_1+\cdots+x_{m-1}$ be an NSG-composition ending with
a maximal part $x_{m-1}=3$. Since $x_i+x_{m-1-i}\geq 3$, we have
$$g=x_1+\cdots+x_{m-1}\geq 3+3(m-2)/2$$
equivalent to the inequality $m\leq 2g/3$.

Computing all initial coefficients of $\tilde C$ up to $\tilde c_g$
can thus be achieved by considering all NSG-compositions of multiplicity
$m\leq 2g/3$ which end with a maximal part of size $3$.

This can be achieved as follows: We fix $m\leq 2g/3$.
We restrict first $x_1,\ldots,x_{m-2}$ to values in $\{1,2\}$.
Since $x_i+x_{m-1-i}\geq 3$, pairs $(x_i,x_{m-1-i})$ of distinct
parts belong to the set $\{(1,2),(2,1),(2,2)\}$. If $m\geq 3$ is odd,
we set $x_{(m-1)/2}=2$. Given such a fixed choice of $x_1,\ldots,x_{m-2}$,
we compute the number
$$f=\{k\leq m-1\ \vert\ x_k=2<x_i+x_{k-i},i=1,\ldots,\lfloor k/2\rfloor\}$$
of corresponding fat edges (representing parts of size $2$ which
can be replaced by parts of size $3$)
in the tree of generic NSG-compositions. Such a choice
yields a total contribution of
$$q^{\sum_{i=1}^{m-1}x_i}(1+q)^f$$
to $\tilde C$.

The computation of $f$ is quadratic in $g$ and does thus not increase
the exponential
complexity which comes from the roughly $3^{g/3}$ possible choices for
the pairs $(x_i,x_{m-1-i})$.

\section{Asymptotic properties for special compositions}

The $k$-th part $x_k$ of a composition
$x_1+\ldots$ with very large sum $g=\sum_j x_j$ with respect to the index $k$,
chosen uniformly among all compositions of $g$ has obviously
an asymptotic limit-distribution{\footnote{Disclaimer:
    Probability theory is a branch of mathematics where I tend
    to be absolutely sure and completely wrong.}}: It is equal to $n$ with
asymptotic probability $2^{-n}$ for $g\rightarrow \infty$.

Such asymptotic limit-distributions for parts exist more generally for
partitions with parts satisfying suitable restrictions
(satisfying some mild hypotheses)
depending at most on the index of parts (the resulting limit-distributions
depend then also on the index of the part under consideration).
An example is given by compositions $x_1+\cdots$ with a non-zero part $x_k$
in $\{1,\ldots,k\}$. Such compositions are enumerated by the sequence
A8930 of \cite{OEIS}.

From an enumerative point of view, one can ask for generating series
of compositions satisfying some restrictions.
The answer, given by
$1/(1-\sum_{a\in \mathcal A}q^a)$ is easy for compositions with all parts
$x_i$ in a common subset $\mathcal A$ of non-zero elements in $\mathbb N$.
More generally, for $x_i$ restricted to non-empty subsets $\mathcal A_i$ of
$\mathbb N\setminus\{0\}$, we get
$$\sum_{n=0}^\infty \prod_{k=1}^n\left(\sum_{a\in\mathcal A_k}q^a\right).$$

\begin{rem} It is possible to consider more generalised compositions
  where parts of equal size can have a finite number of
  different 'colours' (depending perhaps on the index of the part).
  We leave the details to the reader.
\end{rem}

Putting restrictions on parts depending not only on the index of parts but
also of all previous parts is more challenging:
Let us consider compositions $x_1+\cdots$ whose $k$-th non-zero part
belongs to some non-empty
subset $\mathcal A(x_1,\ldots,x_{k-1})$ of strictly positive
integers. The existence of limit-distributions for parts
$x_1,x_2,\ldots$ is probably
no longer easy to decide: $\mathcal A(x_1,\ldots,x_{k-1})=
\{1,2,\ldots,x_{k-1}\}$ (and no restriction for the first part)
leads for example to partitions
(compositions with decreasing parts)
which have no asymptotic limit-distribution. Ordering parts
of partitions in increasing order (by considering compositions
with finitely many parts such that $x_1\leq x_2\leq x_3\ldots$)
yields to a trivial asymptotic distribution: $x_k=1$ asymptotically
for almost all such compositions of sufficiently large integers.

A pseudo-example is given by $x_1=1$ (if it exists) and $x_k\in\{1,x_{k-1}+1\}$.
Regrouping suitable terms of such compositions yields a bijection
with compositions 
having all parts $x_i$ in the set $\{1,1+2,1+2+3,\ldots,{k\choose 2},\ldots\}$
of triangular numbers.

Since the set $\mathcal A(x_1,\ldots,x_{k-1})$
of possible values for $x_k$ depends
on $x_1,\ldots,x_{k-1}$, perhaps existing
limit-distributions for different parts $x_i,x_j$
(in compositions with $x_k$ in $\mathcal A(x_1,\ldots,x_{k-1})$)
are no longer independent and we can also consider asymptotic probabilities
that a given random compositions starts with $x_1+\ldots+x_k$.
The corresponding asymptotic
probabilities (and related quantities) will be discussed in
the next Section for generic NSG-compositions.

An example where an asymptotic limit-distribution
exists is given by generic NSG-compositions:
$\mathcal A(x_1,\ldots,x_{k-1})=\{1,2\}$ if $x_i=x_{k-i}=1$ for
some $i<k$ and $\mathcal A(x_1,\ldots,x_{k-1})=\{1,2,3\}$
otherwise. We will give a few more details below.

An asymptotic limit distribution should also exist for
\begin{align}\label{Akmcond}
\mathcal A(x_1,\ldots,x_{k-1})&=\{1,2,\ldots,
                                \min_{i,1\leq i<k}x_i+x_{i-k}\},
\end{align}
i.e. for composition satisfying only the NSG-inequalities
given by the first line of (\ref{fundeqsg}). The methods of this paper
give however no rigorous proof for the existence of limit-distributions
in this case since our proofs for compositions with parts of size
larger than $3$ involve also right factors in
pivot-factorisations.

In particular, the numbers $\nu_g$
(defined as the number of compositions $x_1+\ldots+x_m$
with $\sum_ix_i=g$ and $x_i+x_j\geq x_{i+j}$ whenever $i+j\leq m$)
of such compositions of $g$ have probably nice asymptotics.
The first $25$ values $\nu_1,\ldots,\nu_{25}$ are
$$\begin{array}{l}1,2,4,7,13,25,43,79,142,254,449,800,1407,2475,4339,7590,13222,\\
23009,39898,69068,119353,205842,354267,608805,1044528.\end{array}\ .$$

A fairly easy example with respect to enumeration is given
by $\mathcal A_1=\{1,2,3,\ldots\}$ and $\mathcal A_k(x_1,\ldots,x_{k-1})=
\{1,\ldots,x_{k-1}+1\}$.
Denoting by $G_k$ the generating series of all such compositions ending
with a last part of size $k$ we have
$$G_k=x^k\left(1+\sum_{i=\max(k-1,1)}^\infty G_i\right)\ .$$
This allows to compute finite series-expansions of $G_1,G_2,\ldots$
by 'bootstrapping'. The generating series for all such compositions
(with $x_i\leq x_{i-1}+1$) is of course defined by $1+\sum_{i=1}^\infty G_i$.
Its coefficients define the series A3116 (by definition) of \cite{OEIS}
and seem to have asymptotics of the form $\gamma\cdot \lambda^n$
with $\gamma=0.52893714\ldots$ and $\lambda=1.7356628\ldots$.

I ignore if parts of such compositions have
asymptotic limit-distributions.

A last example with probably rather small exponential growth
is given by compositions with arbitrary $x_1$ and with
$x_k\in\{1,1+x_{k-1},1+x_{k-1}+x_{k-2},\ldots,1+\sum_{j=i}^{k-1}x_j,\ldots,
1+\sum_{j=1}^{k-1}x_j\}$
(which forbids two consecutive parts of identical size larger than $1$).

\section{Probabilities related to generic NSG-compositions}\label{sectprobas}

Given a composition $\mathbf x=x_1+x_2+\cdots+x_k$, we denote by
$P_g(\mathbf x)$ the proportion of NSG-compositions of genus $g
\geq \sum_{i=1}^k x_i$ starting with
$\mathbf x$ (among all NSG-compositions of genus $g$).
This proportion tends to a limit $P(\mathbf x)=\lim_{g\rightarrow\infty}
P_g(\mathbf x)$ defining a natural probability law on generic compositions with $k$ parts in $\{1,2,3\}$.
The limit-probability satisfies
$$P(\mathbf x)=P(\mathbf x+1)+P(\mathbf x+2)+P(\mathbf x+3)$$
and is non-zero on a composition $\mathbf x=x_1+\cdots+x_n$
if and only if all parts $x_1,\ldots,x_n$ are elements of $\{1,2,3\}$
and $x_{i+j}\leq 2$ whenever $x_i=x_j=1$. Equivalently, $P(\mathbf x)$
is non-zero if and only if $\mathbf x$ encodes a path starting
at the root
in the combinatorial model of Section \ref{sectcombtree} where parts
of size $1$ correspond to left edges, parts of size $2$ to (slim or fat)
right edges and parts of size $3$ to fat right edges.
For simplicity, we call compositions $\mathbf x$
such that $P(\mathbf x)>0$ henceforth
\emph{generic NSG-compositions}.

The limit-probability $P$ encodes some aspects of the generic behaviour
of (uniformly distributed) NSG-compositions of large genus.

An interesting feature of these probabilities $P$ is the following result:
\begin{prop}\label{propP2P3} Given an arbitrary generic NSG-composition $\mathbf x$,
  we have
  $$\frac{P(\mathbf x+3)}{P(\mathbf x+2)}\in\{0,\omega^{-1}\}\ .$$
\end{prop}

\begin{proof} This ratio is obviously $0$ if no additional part of
  size $3$ can be appended to $\mathbf x$. Otherwise
  we get a bijection between NSG-compositions of genus $g$
  ending with a last part $2$ (by appending an additional part $2$)
  and some NSG-composition of genus $g+1$ ending with a last
  part $3$ (by appending an additional part $3$).

  Theorem \ref{thmupperboundng} ends the proof.
\end{proof}

Proposition \ref{propP2P3} shows that
\begin{align}\label{P1overP123}
  \rho(\mathbf x)&=\frac{P(\mathbf x+1)}{P(\mathbf x+1)+P(\mathbf x+2)+P(\mathbf x+3)}
\end{align}
are essentially
the only interesting values: The
combinatorics of $\mathbf x$ determines
if $P(\mathbf x+3)=0$. Proposition \ref{propP2P3} determines then
$P(\mathbf x+1),P(\mathbf x+2)$ and $P(\mathbf x+3)$ uniquely in
terms of $\rho(\mathbf x)$ and $P(\mathbf x)=P(\mathbf x+1)+P(\mathbf x+2)+P(\mathbf x+3)$.

We have obviously $\rho(\mathbf x)\rightarrow
\omega^{-1}$ for most generic NSG-compositions $\mathbf x$ with multiplicity
(or genus) tending to $\infty$.

Using the tree model of
Section \ref{sectcombtree} and identifying infinite geodesics starting at the root of the the binary tree with binary expansions of
elements in $[0,1]$, the probability laws $P$
correspond to a continuous distribution function on $[0,1]$.

A random-variable related to these probabilities is the asymptotic number of
maximal parts equal to $3$ in generic NSG-compositions:
Let $A_g(n)$ be the proportion of generic NSG-compositions of genus $g$
having exactly $n$
parts of size $3$. We get asymptotic limit-probabilities
$$A(n)=\lim_{g\rightarrow\infty}A_g(n)=\frac{\tau(n)}{1+\tilde C(\omega^{-1})}$$
for $\tau(0)=1$ and $\tau(k)$ defined by (\ref{formtauk}) for $k\geq 1$.
The number $A(n)$ is the asymptotic proportion
of $NSG$-composition
with $n$ parts of
size $3$ among all NSG-compositions of very large genus.
(It is not necessarily to require that the $n$ parts of size $3$
are of maximal size: NSG-compositions with larger parts can be
neglected when considering asymptotics.)
Particularly interesting is the value of $A(0)$ (i.e. the proportion
of NSG-compositions having all parts in $\{1,2\}$) since we have obviously
the identity
$$1+\tilde C(\omega^{-1})=\frac{1}{A(0)}$$
linking the probability $A(0)$ to the value
of the constant $C=\frac{5+\sqrt{5}}{10A(0)}$, cf. Formula
(\ref{formulaforC}) in Theorem \ref{thmupperboundng}.

Similarly, we have
$$\frac{A(1)}{A(0)}=\frac{1+\omega^{-2}}{1-(2\omega^{-3}+\omega^{-4})}\omega^{-3}
=\frac{1}{\omega}+\frac{1}{\omega^3}$$
since $\frac{1+q^2}{1-(2q^3+q^4)}q^3$ is the generating series for all
NSG-compositions $x_1+\cdots+x_{m-1}$ with $x_{m-1}=3$ and
$x_1,\ldots,x_{m-2}$ in $\{1,2\}$
(satisfying $x_i+x_{m-1-i}\geq 3$), see Proposition
\ref{propCtildeuniqemaxpart}.

\begin{rem} Typical NSG-compositions
  of high genus have only very few parts of size $3$. They behave thus
  very differently from $\tilde C$-typical contributions to $\tilde C$
  which should have many maximal parts.
\end{rem}

A similar random-variable 
(on $\mathbb Z\setminus\{0\}$) defined by generic NSG-compositions
is given by 
$f-2m$ (for $f$ the Frobenius number and $m$ the multiplicity).

Last parts of generic NSG-compositions have also an asymptotic limit
distribution, simply given by independent Bernoulli distributions
yielding final parts of size $1$ with asymptotic
probability $\omega^{-1}$ and final
parts of
size $2$ with asymptotic probability $\omega^{-2}$. In particular, we
have most of the time
\begin{align}\label{limrho}
  \lim_{\vert \mathbf x\vert,\ P(\mathbf x)>0}\rho(\mathbf x)&=\omega^{-1}
\end{align}
with $\vert\mathbf x\vert$ denoting the length (number of summands)
of $\mathbf x$. (Exceptions can occur if the repartition of parts of
size $1$ in the first half of $\mathbf x$ is atypical.)

More precisely, a generic NSG-composition of large genus $g$ 
has typically multiplicity $\frac{5+\sqrt{5}}{10}g+O(\sqrt{g})$.
It consists of $g/\sqrt{5}+O(\sqrt{g})$ parts of size $1$,
of $\frac{5-\sqrt{5}}{10}g+O(\sqrt{g})$ parts of size $2$
and of a small number (given by the random variable
$A(n)$ considered above) of parts $3$ 
among its initial parts.

\begin{rem}
One can also consider probability laws 
corresponding to largest gaps in generic semigroups.
The probability that an element $f-a$ at distance $a$
of the Frobenius element (largest gap) $f=\max(\mathbb N\setminus S)$
of a generic numerical 
semigroup $S$ does not belong to $S$ tends to $\omega^{-2}$
for $a\rightarrow\infty$.
\end{rem}

\subsection{A toy generator for
  NSG-compositions using unfair coin tosses}

A naive way to generate NSG-composition with given multiplicity
$m$ is to choose generators in $\{m+1,m+2,\ldots\}$ independently
with probability $\lambda$ in $(0,1)$. This results
in NSG-compositions with maximum $2$ or $3$ having Frobenius numbers
close to $2m$ and genus $(2-\lambda)m+O(\sqrt{m})$.

For $\lambda=\frac{\sqrt{5}-1}{2}=\omega^{-1}$ this should
lead to more or less uniform random NSG-compositions
for large $m$, as suggested by
(\ref{P1overP123}) and (\ref{limrho}).

The corresponding probabilities $\tilde P_\lambda(\mathbf x)$
are easy to compute: Given $\mathbf x=x_1+\cdots+x_k$ with $x_1,\ldots,x_k$
in $\{1,2,3\}$,
the asymptotic probability (for $m\rightarrow\infty$)
to generate a NSG-composition starting with
$\mathbf x$ can be computed as follows:
$\tilde P_\lambda(\mathbf x)=0$ if and only if
there exists $i,j$ (not necessarily distinct)
with $i+j\leq k$ such that $x_i=x_j=1$ and $x_{i+j}=3$ (i.e. if
$\mathbf x$ is not a generic NSG-composition).
Otherwise, the probability $\tilde P_\lambda(\mathbf x)$ is a product of
$k$ factors in $\{\lambda,(1-\lambda),\lambda(1-\lambda),(1-\lambda)^2\}$
defined as follows:

\begin{itemize}
\item{} Every summand $x_i=1$ contributes a factor $\lambda$.
\item{} A summand $x_i=2$ contributes a factor $(1-\lambda)$
  if there exists $j<i$ such that $x_j=x_{i-j}=1$.
  It contributes a factor $\lambda(1-\lambda)$ otherwise.
\item{} A summand $x_i=3$ contributes a factor $(1-\lambda)^2$.
\end{itemize}

It is easy to check that the probabilities $\tilde P_{\omega^{-1}}$
defined in this way satisfy Proposition \ref{propP2P3}.
We have moreover
\begin{align}\label{tP1overtP123}
  \frac{\tilde P_{\omega^{-1}}(\mathbf x+1)}{\tilde P_{\omega^{-1}}(\mathbf x+1)+\tilde P_{\omega^{-1}}(\mathbf x+2)+\tilde P_{\omega^{-1}}(\mathbf x+1)}
  &=\omega^{-1},
\end{align}
cf. (\ref{P1overP123}) and (\ref{limrho}).

NSG-compositions sampled in this way (for a given fixed $\lambda$
in $(0,1)$) have typically only a small number of summands $3$.
Setting
$$\mu_h(x_1+\cdots+x_k)=\sum_{i,x_i=3}i^h\ ,$$
the limit-expectancy of $\mu_h$ (with respect to $\tilde P_\lambda(\mathbf x)$
is easy to compute and is given by
$$\mu_h=(1-\lambda)^2\sum_{n=0}^\infty (1-\lambda^2)^n\left((2n+1)^h+(1-\lambda)
  (2n+2)^h\right)$$
(which is a rational function and can be rewritten in terms of dilogarithms).
Indeed, a part $x_{2n+1}$, respectively $x_{2n+2}$, can be equal to
$3$ (with probability $(1-\lambda)^2$) if and only if
$\{x_i,x_{2n+1-i}\}\not=\{1\}$ (which happens with probability
  $1-\lambda^2$), respectively $\{x_i,x_{2n+2-i}\}\not=\{1\}$
  which happens with probability $1-\lambda$ in the case $i=n+1$.
  In both cases, there are $n$ such distinct pairs
  $\{i,2n+1-i\}$, respectively $\{i,2n+2-i\}$ containing two indices.
  Independency of choices among
  parts of size $2,3$ whenever possible leads to the formula.

In particular, the expected asymptotic
number of parts of
size $3$ in large random NSG-compositions sampled accordingly to
$\tilde P_\lambda(\mathbf x)$ is equal to
\begin{align}\label{formulaformu0}
\mu_0&=\frac{(2-\lambda)(1-\lambda)^2}{\lambda^2}
\end{align}
which evaluates to
$5-2\sqrt{5}=0.52786\ldots$ at $\lambda=\omega^{-1}$.

The following variation generates NSG-compositions of genus $g$
accordingly to the law $\tilde P_\lambda$:
We do not fix $m$ but consider it as an unknown,
  to be fixed later. We add generators $m+i$ with independent uniform
  probability $\lambda$ and generators $2m+j$ again with uniform
  independent probability $\lambda$ (most of them will be of the
  form $(m+i_1)+(m+i_2)$ for generators $m+i_1,m+i_2$ already chosen).
  This defines the beginning of an NSG-composition $x_1+x_2+\cdots$.
  Stop if $x_1+\cdots+x_k\geq g$. Set $m=k+1$ and accept
  $\mathbf x=x_1+\cdots+x_k$ if $x_1+\cdots+x_k=g$. Reject it and restart
  if $x_1+\cdots+x_k>g$ (which happens asymptotically with
  probability $(1-\lambda)/2$). The expected multiplicity
  $m$ is asymptotically given by
  $$\frac{g}{2-\lambda}-\mu_0$$
  with
  $\mu_0$ given by (\ref{formulaformu0}) denoting
  the asymptotic expectation for the number of parts of size $3$.

  I ignore how to compute the asymptotic probability
  $\tilde A_\lambda(n)=\lim_{g\rightarrow\infty}\tilde A_{\lambda,g}(n)$
  with $A_{\lambda,g}(n)$ denoting the proportion of NSG-compositions
  of genus $g$ sampled as above which have $n$ parts of size $3$.
  
  NSG-compositions generated by this algorithm with $\lambda=\omega^{-1}$
  are
  not (asymptotically) uniformly sampled. Indeed, the expected number of
  parts of size $3$ in uniformly sampled
  NSG-compositions is asymptotically equal to
  $$\frac{\sum_{n=1}^\infty n\tau(n)}{1+\sum_{n=1}^\infty \tau(n)}\geq
  \frac{\tau(1)+2\tau(2)}{1+\tau(1)+\tau(2)}=0.909389\ldots$$
  which is larger than the corresponding expectation $\mu_0$
  for the toy generator with $\lambda=\omega^{-1}$.

\noindent Roland BACHER, 

\noindent Univ. Grenoble Alpes, Institut Fourier, 

\noindent F-38000 Grenoble, France.
\vskip0.5cm
\noindent e-mail: Roland.Bacher@univ-grenoble-alpes.fr

\end{document}